\documentclass[a4paper, 12pt]{amsart}
\usepackage{graphicx}
\usepackage{amsmath,amssymb}

\newtheorem{thm}{Theorem}[section]
\newtheorem{cor}[thm]{Corollary}
\newtheorem{lem}[thm]{Lemma}
\newtheorem{prop}[thm]{Proposition}

\theoremstyle{definition}

\theoremstyle{remark}
\newtheorem{rem}[thm]{Remark}
\theoremstyle{example}

\theoremstyle{conjecture}

\numberwithin{equation}{section}

\newcommand{\RR}{{\mathbb R}}
\newcommand{\CC}{{\mathbb C}}

\newcommand{\NN}{{\mathbb N}}

\newcommand{\calF}{{\mathcal F}}
\newcommand{\calI}{{\mathcal I}}
\newcommand{\calJ}{{\mathcal J}}

\newcommand{\calO}{{\mathcal O}}

\newcommand{\calW}{{\mathcal W}}

\begin{document}

\title[Weighted composition operators]{Weighted composition operators between Fock spaces in several variables}

\author{Pham Trong Tien$^1$ \& Le Hai Khoi$^2$}%
\address{(Tien) Department of Mathematics, Mechanics and Informatics, Hanoi University of Science, VNU, 334 Nguyen Trai, Hanoi, Vietnam}%
\email{phamtien@vnu.edu.vn, phamtien@mail.ru}%%

\address{(Khoi) Division of Mathematical Sciences, School of Physical and Mathematical Sciences, Nanyang Technological University (NTU),
637371 Singapore}%
\email{lhkhoi@ntu.edu.sg}%%

\thanks{$^1$ Supported in part by NAFOSTED grant No. 101.02-2017.313. \\
\indent $^2$ Supported in part by MOE's AcRF Tier 1 M4011724.110 (RG128/16)}

\subjclass[2010]{47B33, 32A15}%

\keywords{Fock space, weighted composition operator, essential norm}%

\date{\today}%

%\thanks{}

%\date{*}%
%\dedicatory{}%
%\commby{}%
% ----------------------------------------------------------------

\begin{abstract}
We obtain criteria for the boundedness and compactness of weighted composition operators between different Fock spaces in $\CC^n$. We also give estimates for essential norm of these operators.
\end{abstract}

\maketitle

\section{Introduction}
Let $\CC^n$ be the $n$-dimensional complex Euclidean space and $\calO(\CC^n)$ the space of entire functions on $\CC^n$ with the usual compact open topology.
For a number $p \in (0,\infty)$, the Fock space $\calF^p(\CC^n)$ consists of all functions $f$ from $\calO(\CC^n)$ for which
$$
\|f\|_{n,p} = \left( \left(\dfrac{p}{2\pi}\right)^n \int_{\CC^n} |f(z)|^p e^{-\frac{p|z|^2}{2}}dA(z) \right)^{\frac{1}{p}} < \infty,
$$
where $dA$ is the Lebesgue measure on $\CC^n$. Furthermore, the space $\calF^{\infty}(\CC^n)$ is defined as follows
$$
\calF^{\infty}(\CC^n) = \left\{ f \in \calO(\CC^n): \|f\|_{n, \infty} = \sup_{z \in \CC^n}|f(z)|e^{-\frac{|z|^2}{2}} < \infty \right\}.
$$
It is well known that $\calF^p(\CC^n)$ with $1 \leq p \leq \infty$ is a Banach space, while for $0 < p < 1$, $\calF^p(\CC^n)$ is a complete metric space with the distance $d(f,g)  = \|f-g\|^p_{n, p}$.

For each $w \in \CC^n$, we define the functions
$$
K_w(z) = e^{\langle z,w \rangle} \text{ and } k_w(z) = e^{\langle z,w \rangle - \frac{|w|^2}{2}}, \quad z \in \CC^n,
$$
where $\langle z, w \rangle = z_1 \overline{w_1} + \cdots + z_n \overline{w_n}$ and $|w| = \sqrt{\langle w, w \rangle}$.
In the case $p = 2$, the functions $K_w$ are reproducing kernels of the Hilbert Fock space $\calF^2(\CC^n)$, i.e., $f(z) = \langle f, K_z \rangle$ for all $f \in \calF^2(\CC^n)$ and $z \in \CC^n$.
Moreover, $\|k_w\|_{n, p} = 1$ for all $w \in \CC^n$ and $0< p \leq \infty$, and $k_w$ converges to $0$ in the space $\calO(\CC^n)$ as $|w| \to \infty$.

The Fock spaces and classical operators on them play an important role in harmonic analysis on the Heisenberg group, partial differential equations and quantum physics (see e.g. \cite{KZ}). From this reason they have been studied intensively in different directions (see, for instance, \cite{BI-12, BC-87, HL-11, IZ-10} for Toeplitz operators, \cite{S-92, W-89} for Hankel operators, \cite{C-12, H-13} for Volterra-type integration operators, \cite{CMS-03, CIK-10, D-14, HK-16, HK-17, T-14, U-07} for (weighted) composition operators).

In this paper we are interested in weighted composition operators between different Fock spaces in $\CC^n$. Note that Ueki \cite{U-07} investigated the boundedness, compactness and essential norm of weighted composition operators $W_{\psi,\varphi}$ on Hilbert Fock spaces $\calF^2(\mathbb C^n)$ in terms of a certain integral transform $B_{\varphi}(|\psi|^2)(z)$. However, these results are quite difficult to use, even for composition operators $C_{\varphi}$, for which Carswell, MacCluer and Schuster \cite{CMS-03} had already provided the criteria for the boundedness and compactness.
Later, Le \cite{T-14} obtained much easier characterizations for the boundedness and compactness of $W_{\psi,\varphi}$ on Hilbert Fock space $\calF^2(\mathbb C)$. Recently, in \cite{TK-17} the authors extended Le's results to $W_{\psi,\varphi}$ acting from one Fock space $\calF^p(\mathbb C)$ to another one $\calF^q(\mathbb C)$ and stated also simpler estimates for essential norm of such operators $W_{\psi,\varphi}$. It should be noted that in \cite{T-14, U-07} the techniques of adjoint operators on Hilbert spaces played an essential role, while in the context of \cite{TK-17} these techniques do not work and a new approach is required. The aim of this paper is to develop the study in \cite{TK-17} for the case of several variables.

The paper is organized as follows. In Section 2 we give some preliminary results about Fock spaces $\calF^p(\CC^n)$, operators defined on them and an extension of \cite[Proposition~2.1]{T-14} to several variables, which plays a crucial role in our study. Section 3 contains main criteria for the boundedness and compactness of $W_{\psi,\varphi}: \calF^p(\CC^n) \to \calF^q(\CC^n)$. In this section, firstly we prove that these bounded weighted composition operators $W_{\psi,\varphi}$ can be induced only by mappings $\varphi(z) = Az + b$ with an $n \times n$ matrix $A$, $\|A\| \leq 1$, and a vector $b$ in $\CC^n$. Then we use the singular value decomposition $A = V\widetilde{A}U$ to reduce the study of $W_{\psi,\varphi}$ to that of $W_{\widetilde{\psi},\widetilde{\varphi}}$ induced by the so-called \textit{normalization} $(\widetilde{\psi},\widetilde{\varphi})$ of $(\psi,\varphi)$, where $\widetilde{\varphi}(z) = \widetilde{A} z + \widetilde{b}$ and $\widetilde{A}$ is a diagonal matrix with $1 \geq \widetilde{a}_{11} \geq \widetilde{a}_{22} \geq ... \geq \widetilde{a}_{ss} \geq \widetilde{a}_{s+1,s+1} = ... = \widetilde{a}_{nn} =0$.
From this we get necessary and sufficient conditions for the boundedness and compactness of $W_{\psi,\varphi}$ in terms of $(\widetilde{\psi},\widetilde{\varphi})$. In the case when $W_{\psi,\varphi}$ acts from a larger Fock space into a smaller one, these properties are equivalent. By this reason Section 4 deals with estimates for essential norm of only operators $W_{\psi,\varphi}: \calF^p(\CC^n) \to \calF^q(\CC^n)$ with $p \leq q$.

\section{Preliminaries}
In this section we give some auxiliary results which will be used throughout the paper.

For each $z = (z_1,...,z_n) \in \CC^n$ and $0 \leq s \leq n$, we denote
$$
z_{[s]} =
\begin{cases}
\emptyset, \quad \quad \quad \ \ \text{ if } s = 0 \\
(z_1,...,z_s), \text{ if } s \neq 0,
\end{cases}
\text{and  }
z'_{[s]} =
\begin{cases}
(z_{s+1},...,z_n), \text{ if } s \neq n\\
\emptyset, \quad \quad \quad \quad \ \ \text{ if } s = n,
\end{cases}
$$
by convention that
$ |z_{[0]}| = 0$ and $|z'_{[n]}| = 0$.

\begin{lem}\label{lem-F}
Let $p \in (0,\infty)$, $b = (b_1,...,b_n) \in \CC^n$, and $f \in \calF^p(\CC^n)$. For each $0 < s < n$ the following statements are true:
\begin{itemize}
\item[(i)] The function $f(b_{[s]},\cdot) \in \calF^p(\CC^{n-s})$ and
$$
\|f(b_{[s]},\cdot)\|_{n-s, p}\; e^{-\frac{\left|b_{[s]}\right|^2}{2}} \leq \|f\|_{n, p}.
$$
\item[(ii)] The function $f(\cdot,b'_{[s]}) \in \calF^p(\CC^s)$ and
$$
\|f(\cdot,b'_{[s]})\|_{s, p}\; e^{-\frac{\left|b'_{[s]}\right|^2}{2}} \leq \|f\|_{n, p}.
$$
\end{itemize}
\end{lem}
\begin{proof} Since proofs of (i) and (ii) are similar, we prove, for example, (i).

Consider the function
$$
F(z_{[s]},z'_{[s]}) = f(b_{[s]}-z_{[s]},z'_{[s]})e^{\langle z_{[s]}, b_{[s]} \rangle - \frac{\left|b_{[s]}\right|^2}{2}}, \ z = (z_{[s]}, z'_{[s]}) \in \CC^n.
$$

Obviously,  for every $z'_{[s]} \in \CC^{n-s}$ fixed, the function $h(z_{[s]}) = |F(z_{[s]},z'_{[s]})|^p$ is a plurisubharmonic function of $z_{[s]}$.
Using the plurisubharmonicity of $h$ and polar coordinates, we have
\begin{align*}
&\int_{\CC^s} h(z_{[s]}) e^{-\frac{p\left|z_{[s]}\right|^2}{2}} dA(z_{[s]}) \\
= & \int_0^{\infty}...\int_0^{\infty} r_1...r_s e^{-\frac{p(r_1^2 + ...+ r_s^2)}{2}} dr_1...dr_s \\
& \qquad \qquad \qquad \qquad \times \int_0^{2\pi}...\int_0^{2\pi}h(r_1e^{i\theta_1},...,r_se^{i\theta_s})d\theta_1...d\theta_s \\
\geq & (2\pi)^s h(0)\int_0^{\infty}...\int_0^{\infty} r_1...r_s e^{-\frac{p(r_1^2 + ...+ r_s^2)}{2}} dr_1...dr_s = \left(\dfrac{2\pi}{p}\right)^s h(0).
\end{align*}
From this it follows that
\begin{align*}
&\left|f(b_{[s]},z'_{[s]})\right|^p e^{-\frac{p\left|b_{[s]}\right|^2}{2}} = \left|F(0,z'_{[s]})\right|^p\\
\leq & \left( \dfrac{p}{2\pi} \right)^{s} \int_{\CC^s} \left|f(b_{[s]}-z_{[s]},z'_{[s]})e^{\langle z_{[s]}, b_{[s]} \rangle - \frac{\left|b_{[s]}\right|^2}{2}}\right|^p e^{-\frac{p\left|z_{[s]}\right|^2}{2}}dA(z_{[s]}) \\
= & \left( \dfrac{p}{2\pi} \right)^{s} \int_{\CC^s} \left|f(b_{[s]}-z_{[s]},z'_{[s]})\right|^p e^{-\frac{p\left|b_{[s]} - z_{[s]}\right|^2}{2}}dA(z_{[s]}) \\
= & \left( \dfrac{p}{2\pi} \right)^{s} \int_{\CC^s} \left|f(z_{[s]},z'_{[s]})\right|^p e^{-\frac{p\left|z_{[s]}\right|^2}{2}}dA(z_{[s]}).
\end{align*}
Therefore,
\begin{align*}
&\|f(b_{[s]},\cdot)\|_{n-s, p}^p\; e^{-\frac{p\left|b_{[s]}\right|^2}{2}} \\
= & \left( \dfrac{p}{2\pi} \right)^{n-s} \int_{\CC^{n-s}} \left|f(b_{[s]},z'_{[s]})\right|^p e^{-\frac{p\left|b_{[s]}\right|^2}{2}} e^{-\frac{p\left|z'_{[s]}\right|^2}{2}}dA(z'_{[s]}) \\
\leq & \left( \dfrac{p}{2\pi} \right)^{n} \int_{\CC^{n-s}}\int_{\CC^s} \left|f(z_{[s]},z'_{[s]})\right|^p e^{-\frac{p\left|z_{[s]}\right|^2}{2}} e^{-\frac{p\left|z'_{[s]}\right|^2}{2}} dA(z_{[s]})dA(z'_{[s]}) \\
= & \left( \dfrac{p}{2\pi} \right)^{n} \int_{\CC^n} |f(z)|^p e^{-\frac{p|z|^2}{2}}dA(z) = \|f\|^p_{n,p},
\end{align*}
which completes the proof.
\end{proof}

The following two lemmas extend the corresponding results of \cite[Corollary 2.8 and Theorem 2.10]{KZ} to the  case of several variables.

\begin{lem} \label{lem-F1}
Let $p \in (0, \infty)$. Then for each function $f \in \calF^p(\CC^{n})$ and $z \in \CC^n$,
$$
|f(z)|e^{-\frac{|z|^2}{2}} \leq \|f\|_{n, p}.
$$
\end{lem}
\begin{proof}
By Lemma \ref{lem-F}, the function $f(\cdot,z'_{[1]}) \in \calF^p(\CC)$ and
$$
\|f(\cdot,z'_{[1]})\|_{1, p}\; e^{-\frac{|z'_{[1]}|^2}{2}} \leq \|f\|_{n, p}.
$$
From this and \cite[Corollary 2.8]{KZ} we have
$$
|f(z)|e^{-\frac{|z|^2}{2}} \leq \|f(\cdot,z'_{[1]})\|_{1, p}\; e^{-\frac{|z'_{[1]}|^2}{2}} \leq \|f\|_{n, p}.
$$
\end{proof}

\begin{lem}\label{lem-Fpq}
For every $ 0 < p < q < \infty$, $\calF^p(\CC^n) \subset \calF^q(\CC^n)$ and the inclusion is continuous. Moreover,
$$
\|f\|_{n, q} \leq \left(  \dfrac{q}{p} \right) ^{\frac{n}{q}} \|f\|_{n, p}, \text{ for all } f \in \calF^p(\CC^n).
$$
\end{lem}
\begin{proof}
For each function $f \in \calF^p(\CC^n)$, by Lemma \ref{lem-F1}, we have
\begin{align*}
\|f\|_{n, q}^q & =  \left( \dfrac{q}{2\pi} \right)^n \int_{\CC^n} |f(z)|^q e^{-\frac{q|z|^2}{2}}dA(z) \\
& = \left( \dfrac{q}{2\pi} \right)^n \int_{\CC^n} |f(z)|^{q-p} |f(z)|^p e^{-\frac{q|z|^2}{2}}dA(z) \\
& \leq \left( \dfrac{q}{2\pi} \right)^n \|f\|_{n, p}^{q-p} \int_{\CC^n} |f(z)|^p e^{-\frac{p|z|^2}{2}}dA(z)
 = \left( \dfrac{q}{p} \right)^n \|f\|_{n, p}^{q}.
\end{align*}
From this the desired results follow.
\end{proof}

Similarly to \cite[Lemmas~2.3 and 2.4]{TK-17}, we can easily prove the following lemmas.

\begin{lem}\label{lem-com}
Let $p, q \in (0, \infty)$, $T$ be a linear continuous operator from $\calO(\CC^n)$ into itself and $T: \calF^p(\CC^n) \to \calF^q(\CC^n)$ be well-defined. Then $T: \calF^p(\CC^n) \to \calF^q(\CC^n)$ is compact if and only if for every bounded sequence $(f_j)_j$ in $\calF^p(\CC^n)$ converging to $0$ in $\calO(\CC^n)$, the sequence $(Tf_j)_j$ converges to $0$ in $\calF^q(\CC^n)$.
\end{lem}

For an arbitrary operator $T: \calF^p(\CC^n) \to \calF^q(\CC^n)$ that would be not defined on $\calO(\CC^n)$, we only get the following result.

\begin{lem}\label{lem-com1}
Let $p, q \in (1, \infty)$. If a linear operator $T: \calF^p(\CC^n) \to \calF^q(\CC^n)$ is compact, then for every sequence $(w^{(j)})_j$ in $\CC^n$ with $|w^{(j)}| \to \infty$ as $j \to \infty$, the sequence $(Tk_{w^{(j)}})_j$ converges to $0$ in $\calF^q(\CC^n)$.
\end{lem}

Let  $\psi: \CC^n \to \CC$ be a nonzero entire function and $\varphi: \CC^n \to \CC^n$ a holomorphic mapping. The \textit{weighted composition operator} $W_{\psi,\varphi}$ induced by $\psi, \varphi$ is defined by $W_{\psi,\varphi}f = \psi \cdot (f \circ \varphi)$. When the function $\psi$ is identically $1$, the operator $W_{\psi,\varphi}$ reduces to the \textit{composition operator} $C_{\varphi}$. The study of (weighted) composition operators have received a special attention of many authors during the past several decades (see \cite{CM, Sha} and references therein). A main problem in the investigation of such operators is to describe operator theoretic properties of  $C_{\varphi}$ and $W_{\psi,\varphi}$ in terms of function theoretic properties of $\varphi$ and $\psi$.

As in \cite{T-14, TK-17}, the following quantities play an important role in this paper:
$$
m_z(\psi,\varphi) = |\psi(z)|e^{\frac{|\varphi(z)|^2 - |z|^2}{2}},\ z \in \CC^n,
$$
and
$$
m(\psi,\varphi) = \sup_{z \in \CC^n} m_z(\psi,\varphi).
$$

The following result extends \cite[Proposition~2.1]{T-14} to the case of several variables and plays a crucial role in our study.

\begin{prop} \label{prop-varphi}
Let $\psi$ be a nonzero entire function on $\CC^n$ and $\varphi: \CC^n \to \CC^n$ a holomorphic mapping such that $m(\psi, \varphi) < \infty$. Then $\varphi(z) = Az + b$, where $A$ is an $n \times n$ matrix with $\|A\| \leq 1$ and  $b$ is an $n \times 1$ vector.
Moreover, if $A$ is a unitary matrix, then
$$
\psi(z)  = \psi(z^0) e^{\langle z - z^0, - A^*b \rangle}, \text{ for all } z \in \CC^n,
$$
where $z^0$ is some point in $\CC^n$ satisfying $\psi(z^0) \neq 0$.
\end{prop}
\begin{proof}
We divide the proof into two steps.

\textbf{Step 1.} Suppose that $\psi(0) \neq 0$. Let $\varphi = (\varphi_1,...,\varphi_n)$, for each $j = 1, 2,...,n$ and $\zeta \in \mathbb S^n = \{ z \in \CC^n: |z| = 1\}$, put
$$
\psi_{\zeta}(\lambda) = \psi(\lambda \zeta) \text{ and } \varphi_{j,\zeta}(\lambda) = \varphi_j(\lambda \zeta),\ \lambda \in \CC.
$$
If $\zeta \in \mathbb S^n$ and $1 \leq j \leq n$, then we have
\begin{align*}
m(\psi_{\zeta},\varphi_{j,\zeta}) = \sup_{\lambda \in \CC} m_{\lambda}(\psi_{\zeta},\varphi_{j,\zeta}) \leq \sup_{\lambda \in \CC} m_{\lambda \zeta} (\psi,\varphi) \leq m(\psi,\varphi) < \infty.
\end{align*}
Since $\psi_{\zeta}(0) = \psi(0) \neq 0$, the last inequality together with \cite[Proposition 2.1]{T-14} implies that $\varphi_{j,\zeta}(\lambda) = a_{j,\zeta} \lambda + b_{j,\zeta}$ with $|a_{j,\zeta}| \leq 1$.

On the other hand, if $\varphi_j(z)$ has homogeneous expansion $\sum_{s = 0}^{\infty} \Phi_s(z)$, then for each $\zeta \in \mathbb S^n$, $\varphi_{j,\zeta}(\lambda) = \sum_{s = 0}^{\infty} \lambda^s \Phi_s(\zeta)$. Thus, $\Phi_s(\zeta) = 0$ for each $\zeta \in \mathbb S^n$ and $s \geq 2$. Hence, for each $s \geq 2$ and $z \in \mathbb C^n \setminus \{0\}$, $\Phi_s(z) = |z|^s \Phi_s(z/|z|) = 0$. That is, $\Phi_s \equiv 0$ for all $s \geq 2$.

Consequently, $\varphi(z) = Az + b$, where $A$ is an $n \times n$ matrix and  $b$ is an $n \times 1$ vector.

We show that $\|A\| \leq 1$ by contradiction. Assume that $\|A\| > 1$, that is, there exists $\zeta \in \mathbb S^n$ such that $|A\zeta| > 1$. Then
\begin{align*}
|\psi_{\zeta}(\lambda)| \leq m(\psi,\varphi) e^{\frac{|\lambda \zeta|^2 - |\lambda A\zeta + b|^2 }{2}} = m(\psi,\varphi) e^{\frac{|\lambda|^2 - |\lambda A\zeta + b|^2 }{2}} \to 0 \text{ as } \lambda \to \infty.
\end{align*}
This means that $\psi_{\zeta} \equiv 0$ on $\CC$, which is a contradiction, since $\psi_{\zeta}(0) \neq 0$.

Moreover, if $A$ is unitary matrix, then it is easy to see that
$$
\psi(z) = \psi(0) e^{- \langle Az, b \rangle} = \psi(0) e^{\langle z, - A^*b \rangle}.
$$

\textbf{Step 2.} Suppose that $\psi(z^0) \neq 0$ for some $z^0 \in \CC^n$.
Put
$$
\varphi_{0}(z) = \varphi(z + z^0) \text{ and } \psi_{0}(z) = \psi(z + z^0) e^{-\langle z, z^0 \rangle}, \ z \in \CC^n.
$$
Then for every $z \in \CC^n$, we have
\begin{align*}
|\psi_{0}(z)|e^{\frac{|\varphi_{0}(z)|^2 - |z|^2}{2}} & = \left|\psi(z + z^0)e^{-\langle z, z^0 \rangle}\right|e^{\frac{\left|\varphi(z+z^0)\right|^2 - |z|^2}{2}}\\
& = \left|\psi(z + z^0)\right| e^{\frac{\left|\varphi(z+z^0)\right|^2 - \left|z+z^0\right|^2}{2}} e^{\frac{\left|z^0\right|^2}{2}}.
\end{align*}
Hence,
\begin{align*}
\sup_{z \in \CC^n}|\psi_{0}(z)|e^{\frac{|\varphi_{0}(z)|^2 - |z|^2}{2}} & = e^{\frac{\left|z^0\right|^2}{2}} \sup_{z \in \CC^n}\left|\psi(z + z^0)\right| e^{\frac{\left|\varphi(z+z^0)\right|^2 - \left|z+z^0\right|^2}{2}} \\
&= e^{\frac{\left|z^0\right|^2}{2}} m(\psi,\varphi) < \infty.
\end{align*}
Since $\psi_0(0) = \psi(z^0) \neq 0$, by Step 1, $\varphi_0(z) = Az + b^0$, where $A$ is an $n \times n$ matrix $A$ with $\|A\| \leq 1$ and $b^0$ is an $n \times 1$ vector.
Then
$$
\varphi(z) = \varphi_0(z - z^0) = A(z - z^0) + b^0 = Az + b \text{ with } b = b^0 - Az^0.
$$
Moreover, if $A$ is a unitary matrix, then, again by Step 1,
$$
 \psi_0(z) = \psi_0(0) e^{\langle z, -A^*b^0 \rangle} \text{ for all } z \in \CC^n.
$$
Therefore,
\begin{align*}
\psi(z) & = \psi_0(z - z^0) e^{\langle z - z^0, z^0 \rangle} = \psi_0(0) e^{\langle z - z^0, - A^*b^0 \rangle }e^{\langle z - z^0, z^0 \rangle} \\
& = \psi(z^0) e^{\langle z - z^0, z^0 - A^*b^0 \rangle } = \psi(z^0) e^{\langle z - z^0, - A^* b \rangle }, \ \forall z \in \CC^n.
\end{align*}
\end{proof}

Particularly, when $\psi \equiv { \rm const }$ on $\CC^n$, by Proposition \ref{prop-varphi} and the proof of \cite[Theorem 1]{CMS-03}, we get the following result.
\begin{cor}\label{cor-varphi}
Let $\varphi: \CC^n \to \CC^n$ be a holomorphic mapping with $m(1, \varphi) < \infty$. Then $\varphi(z) = Az + b$, where $A$ is an $n \times n$ matrix with $\|A\| \leq 1$ and  $b$ is an $n \times 1$ vector. Moreover, if $|A\zeta| = |\zeta|$ for some $\zeta$ in $\CC^n$, then $\langle A\zeta, b \rangle =0$.
\end{cor}
%--------------------------------------------------------
\section{Boundedness and Compactness}
In this section we study the boundedness and compactness of weighted composition operators.
Proposition \ref{prop-varphi} give us the following necessary condition.

\begin{prop} \label{prop-nec}
Let $p, q \in (0,\infty)$. If the operator $W_{\psi, \varphi}: \calF^p(\CC^n) \to \calF^q(\CC^n)$ is bounded, then $\psi \in \calF^q(\CC^n)$ and $m(\psi,\varphi) < \infty$.

In this case, $\varphi(z) = Az + b$, where $A$ is an $n \times n$ matrix $A$ with $\|A\| \leq 1$ and $b$ is an $n \times 1$ vector.
\end{prop}
\begin{proof}
Obviously, $\psi = W_{\psi,\varphi}(1) \in \calF^q(\CC^n)$. Moreover, for every $w, z \in \CC^n$, by Lemma \ref{lem-F1} and the fact that $\|k_w\|_{n,p} = 1$, we have
\begin{align*}
\|W_{\psi, \varphi}\| & \geq \|W_{\psi, \varphi}k_w\|_{n, q} \geq |W_{\psi,\varphi}k_w(z)|e^{-\frac{|z|^2}{2}} \\
& = |\psi(z) e^{\langle \varphi(z),w \rangle - \frac{|w|^2}{2}}|e^{-\frac{|z|^2}{2}}.
\end{align*}
In particular, with $w = \varphi(z)$, the last inequality becomes
$$
\|W_{\psi, \varphi}\| \geq |\psi(z)| e^{\frac{|\varphi(z)|^2 - |z|^2}{2}}, \ \forall z \in \CC^n,
$$
which implies that $m(\psi, \varphi) \leq \|W_{\psi, \varphi}\| < \infty$.
Hence, by Proposition \ref{prop-varphi}, $\varphi(z) = Az + b$ with some $n \times n$ matrix $A$, $\|A\| \leq 1$, and $n \times 1$ vector $b$.
\end{proof}

As we see below, in general, the necessary condition in Proposition \ref{prop-nec} is not sufficient for the boundedness of the operator $W_{\psi, \varphi}: \calF^p(\CC^n) \to \calF^q(\CC^n)$. Nevertheless, this condition allows us to be only interested in those operators $W_{\psi, \varphi}$ which are induced by nonzero entire functions $\psi \in \calF^q(\CC^n)$ and
mappings $\varphi(z) = Az + b$, where $A$ is an $n \times n$ matrix with $\|A\| \leq 1$ and $b$ is an $n \times 1$ vector.

Firstly, we consider the trivial case when $A$ is the zero matrix, i.e. $\varphi(z) = b$ for all $z \in \CC^n$. In this case the criterion is rather easy to prove.

\begin{prop} \label{prop-zero}
Let $p, q \in (0, \infty)$. Suppose that $\psi$ is a nonzero entire function on $\CC^n$ and $\varphi(z) \equiv b$ on $\CC^n$ with an $n \times 1$ vector $b$. Then the following statements are equivalent:
\begin{itemize}
\item[(i)] $W_{\psi, \varphi}: \calF^p(\CC^n) \to \calF^q(\CC^n)$ is bounded;
\item[(ii)] $W_{\psi, \varphi}: \calF^p(\CC^n) \to \calF^q(\CC^n)$ is compact;
\item[(iii)] $\psi \in \calF^q(\CC^n)$.
\end{itemize}
In this case,
$$
\|W_{\psi, \varphi}\| = e^{\frac{|b|^2}{2}}\|\psi\|_{n, q}.
$$
\end{prop}
\begin{proof} (ii) $\Longrightarrow$ (i) is obvious and (i) $\Longrightarrow$ (iii) immediately follows from Proposition \ref{prop-nec}. We need to prove only (iii) $\Longrightarrow$ (ii).

Suppose that $\psi \in \calF^q(\CC^n)$.
Obviously, $W_{\psi,\varphi}(f)(z) = \psi(z)f(b)$ for all $f \in \calO(\CC^n)$.
Then, by Lemma \ref{lem-F1}, for all $f \in \calF^p(\CC^n)$,
\begin{align*}
\|W_{\psi,\varphi}(f)\|_{n, q} = \|\psi\|_{n, q} |f(b)| \leq \|\psi\|_{n, q} \|f\|_{n, p}\; e^{\frac{|b|^2}{2}}.
\end{align*}
This means that $W_{\psi, \varphi}: \calF^p(\CC^n) \to \calF^q(\CC^n)$ is bounded and
$$
\|W_{\psi,\varphi}\| \leq e^{\frac{|b|^2}{2}}\|\psi\|_{n, q}.
$$
Moreover, $W_{\psi, \varphi}$ has rank $1$, then it is compact.

On the other hand,
\begin{align*}
\|W_{\psi,\varphi}(k_b)\|_{n, q} = \|\psi\|_{n, q} |k_b(b)| = \|\psi\|_{n, q} e^{\frac{|b|^2}{2}},
\end{align*}
which gives
$$
\|W_{\psi, \varphi}\| = e^{\frac{|b|^2}{2}}\|\psi\|_{n, q}.
$$
\end{proof}

The case $A \not \equiv 0$ is much more difficult, because $A$ is not necessarily diagonal or invertible. In order to overcome this difficulty we make use of the so-called singular value decomposition of an $n \times n$ matrix whose proof can be found in \cite[Theorem~2.6.3]{HR-90}.

\begin{lem}
If $A$ is an $n \times n$ matrix of rank $s$, then $A$ can be written as $A = V \widetilde{A} U$, where $V, U$ are $n \times  n$ unitary matrices, and $\widetilde{A}$ is a diagonal matrix $(\widetilde{a}_{ij})$ with $\widetilde{a}_{11} \geq \widetilde{a}_{22} \geq ... \geq \widetilde{a}_{ss} \geq \widetilde{a}_{s+1, s+1} = ... = \widetilde{a}_{nn} = 0$. The $\widetilde{a}_{ii}$ are the non-negative square roots of the eigenvalues of $AA^*$; if we require that they are listed in decreasing order, then $\widetilde{A}$ is uniquely determined from $A$.
\end{lem}

Let $\calW_q$ be the set of all pairs $(\psi,\varphi)$ consisting of a nonzero entire function $\psi$ in $\calF^q(\CC^n)$ and a mapping $\varphi(z) = Az + b$ with a nonzero $n \times n$ matrix $A$ satisfying $0 < \|A\| \leq 1$ and an $n \times 1$ vector $b$.

We denote by $\mathcal V_{q,s}$ the subset of $\calW_q$ consisting of all pairs $(\psi,\varphi)$ in $\calW_q$ with $\varphi(z) = Az + b$, where $A$ is a diagonal matrix $(a_{ij})$ of $\textnormal{rank} A = s > 0$ and
$$
1 \geq a_{11} \geq a_{22} \geq ... \geq a_{ss} \geq a_{s+1, s+1} = ... = a_{nn} = 0.
$$
Note that for each $(\psi,\varphi)$ in $\mathcal V_{q,s}$ and $f \in \calO(\CC^n)$, we have
\begin{align}\label{eq-normf}
\|W_{\psi,\varphi}f\|_{n, q} =& \left( \left( \dfrac{q}{2\pi} \right)^n \int_{\CC^n} |\psi(z)|^q |f(\varphi(z))|^q e^{-\frac{q|z|^2}{2}}dA(z)\right)^{\frac{1}{q}}\\ \nonumber
=& \left( \left( \dfrac{q}{2\pi} \right)^s \int_{\CC^s}  |f(\varphi(z))|^q e^{-\frac{q\left|z_{[s]}\right|^2}{2}} \|\psi(z_{[s]},\cdot)\|^q_{n-s,q} dA(z_{[s]})\right)^{\frac{1}{q}}.
\end{align}
In view of this, we can characterize the boundedness and compactness for $W_{\psi,\varphi}$ induced by $(\psi,\varphi)$ in $\mathcal V_{q,s}$ in terms of the following quantities:
\begin{align*}
\ell_{z_{[s]}}(\psi,\varphi) = e^{\frac{\left|\varphi(z)\right|^2 - \left|z_{[s]}\right|^2}{2}} \|\psi(z_{[s]},\cdot)\|_{n-s, q}, \ \ z_{[s]} \in \CC^s,
\end{align*}
and
$$
\ell(\psi,\varphi) = \sup_{z_{[s]} \in \CC^s} \ell_{z_{[s]}}(\psi, \varphi),
$$
where we consider $\|\psi(z_{[s]},\cdot)\|_{n-s, q} = |\psi(z)|$ if $s = n$ (i.e. $A$ is invertible), and in this case $\ell_z(\psi,\varphi) = m_z(\psi,\varphi)$.

For weighted composition operators $W_{\psi,\varphi}$ induced by $(\psi,\varphi)$ in $\mathcal W_{q}$ with $\varphi(z) = Az + b$ and $\textnormal{rank} A = s$, we may reduce the study to that of some
$W_{\widetilde{\psi},\widetilde{\varphi}}$ induced by $(\widetilde{\psi},\widetilde{\varphi})$ in $\mathcal V_{q,s}$ by the following scheme.

Suppose that the singular value decomposition of $A$ is $V \widetilde{A} U$ and define a new pair $(\widetilde{\psi}, \widetilde{\varphi})$ as follows:
$$
\widetilde{\psi}(z) = \psi(U^*z) \text{ and } \widetilde{\varphi}(z) = \widetilde{A} z + \widetilde{b},\ z \in \CC^n, \text{ where } \widetilde{b} = V^*b.
$$
We call $(\widetilde{\psi}, \widetilde{\varphi})$ a \textit{normalization} of $(\psi, \varphi)$ with respect to the singular value decomposition $A = V \widetilde{A} U$ (briefly, \textit{normalization} of $(\psi, \varphi)$).

It is easy to see that $(\widetilde{\psi}, \widetilde{\varphi}) \in \mathcal V_{q,s}$. Moreover, we have the following result.

\begin{prop} \label{prop-equiv}
Let $p, q \in (0, \infty)$ and $(\psi, \varphi)$ be a pair in $\mathcal W_q$ and $(\widetilde{\psi}, \widetilde{\varphi})$ its normalization. Then the operator $W_{\psi,\varphi}: \calF^p(\CC^n) \to \calF^q(\CC^n)$ is bounded (respectively, compact) if and only if $W_{\widetilde{\psi},\widetilde{\varphi}}: \calF^p(\CC^n) \to \calF^q(\CC^n)$ is bounded (respectively, compact). Moreover, they have the same norm.
\end{prop}
\begin{proof}
For an $n \times n$ matrix $U$, we put
$$
C_U(f) = f \circ U,\ f \in \calO(\CC^n).
$$
Obviously, if $U$ is a unitary matrix, then $C_U$ is invertible on every Fock space $\calF^p(\CC^n)$ with $(C_U)^{-1} = C_{U^*}$ and $\|C_Uf\|_{n, p} = \|f\|_{n, p}$ for all $f \in \calF^p(\CC^n)$.

For every $f \in \calO(\CC^n)$ and $z \in \CC^n$, we have
\begin{align*}
C_U W_{\widetilde{\psi},\widetilde{\varphi}} C_V(f)(z) &= C_U W_{\widetilde{\psi},\widetilde{\varphi}} (f\circ V)(z) = C_U \left( \widetilde{\psi}\cdot(f\circ V \circ \widetilde{\varphi}) \right)(z) \\
 & = \left( \widetilde{\psi} \circ U (z) \right) (f\circ V \circ \widetilde{\varphi} \circ U) (z) \\
&= \psi(z) f(V \widetilde{A} U z + V\widetilde{b}) = \psi(z) f(Az+b)\\
& = W_{\psi,\varphi}(f)(z),
\end{align*}
which means that $W_{\psi,\varphi} = C_U W_{\widetilde{\psi},\widetilde{\varphi}} C_V$, and hence, $W_{\widetilde{\psi},\widetilde{\varphi}} = C_{U^*} W_{\psi,\varphi} C_{V^*}$ on $\calO(\CC^n)$.

From these equalities the assertions follow.
\end{proof}

In view of Proposition \ref{prop-equiv}, we can have criteria for the boundedness and compactness of the weighted composition operator $W_{\psi, \varphi}$ induced by a pair $(\psi, \varphi) \in \calW_q$ in terms of its normalization $(\widetilde{\psi}, \widetilde{\varphi})$, more precisely, in terms of $\ell_{z_{[s]}}(\widetilde{\psi}, \widetilde{\varphi})$. Before doing this, we state some properties of $(\widetilde{\psi}, \widetilde{\varphi})$ and $\ell_{z_{[s]}}(\widetilde{\psi}, \widetilde{\varphi})$.

\begin{lem}\label{lem-as}
For each pair $(\psi, \varphi) \in \calW_q$ with $\varphi(z) = Az + b$ and $\textnormal{rank} A = s$, the following properties are true.
\begin{itemize}
\item[(a)] $m_{z}(\psi, \varphi) = m_{Uz}(\widetilde{\psi}, \widetilde{\varphi})$ for all $z \in \CC^n$, and $m(\psi, \varphi) = m(\widetilde{\psi}, \widetilde{\varphi})$.
In particular, if $A$ is invertible, then $m_{z}(\psi, \varphi) = \ell_{Uz}(\widetilde{\psi}, \widetilde{\varphi})$ for all $z \in \CC^n$, and $m(\psi, \varphi) = \ell(\widetilde{\psi}, \widetilde{\varphi})$.
\item[(b)] $m(\widetilde{\psi}, \widetilde{\varphi}) \leq \ell(\widetilde{\psi}, \widetilde{\varphi})$.
\end{itemize}
\end{lem}
\begin{proof}
(a) For each $z \in \CC^n$, by the definition of $(\widetilde{\psi}, \widetilde{\varphi})$,
\begin{align*}
m_{Uz}(\widetilde{\psi}, \widetilde{\varphi}) &= \left|\widetilde{\psi}(Uz)\right| e^{\frac{\left|\widetilde{\varphi}(Uz)\right|^2 - \left|Uz\right|^2}{2}} = |\psi(z)| e^{\frac{\left|\widetilde{A}Uz + \widetilde{b}\right|^2 - |z|^2}{2}} \\
& = |\psi(z)| e^{\frac{\left|V^*Az + V^*b\right|^2 - |z|^2}{2}} = |\psi(z)| e^{\frac{|Az + b|^2 - |z|^2}{2}} = m_z(\psi,\varphi).
\end{align*}
From this it follows that $m_{z}(\psi, \varphi) = \ell_{Uz}(\widetilde{\psi}, \widetilde{\varphi})$ for all $z \in \CC^n$ whenever $A$ is invertible.

(b) For each pair $(\psi, \varphi) \in \mathcal V_{q,s}$, by the definition of $\ell_{z_{[s]}}(\psi,\varphi)$ and Lemma \ref{lem-F1}, for each $z \in \CC^n$ we have
\begin{align*}
\ell_{z_{[s]}}(\psi,\varphi) & = e^{\frac{|\varphi(z)|^2 - \left|z_{[s]}\right|^2}{2}} \|\psi(z_{[s]},\cdot)\|_{n-s, q} \\
& \geq e^{\frac{|\varphi(z)|^2 - \left|z_{[s]}\right|^2}{2}} |\psi(z)| e^{-\frac{\left|z'_{[s]}\right|^2}{2}} = m_z(\psi,\varphi).
\end{align*}
It implies that $m(\psi, \varphi) \leq \ell(\psi,\varphi)$ for all $(\psi, \varphi) \in \mathcal V_{q,s}$.

From this it follows that $m(\widetilde{\psi}, \widetilde{\varphi}) \leq \ell(\widetilde{\psi}, \widetilde{\varphi})$ for all $(\psi, \varphi) \in \mathcal W_{q}$.
\end{proof}

Clearly, for each pair $(\psi,\varphi) \in \calW_q$, its normalization $(\widetilde{\psi}, \widetilde{\varphi})$ is not unique and depends on the unitary factors $V$ and $U$ in the singular value decomposition $A = V \widetilde{A} U$. But the quantity $\ell_{z_{[s]}}(\widetilde{\psi}, \widetilde{\varphi})$ is "unique" in the following sense.

\begin{lem} \label{lem-UV}
Let $(\widehat{\psi},\widehat{\varphi})$ be another normalization of $(\psi,\varphi) \in \calW_q$ with respect to the singular value decomposition $A = \widehat{V} \widetilde{A} \widehat{U}$.
Then there is an $s \times s$ unitary matrix $H$ such that $\ell_{z_{[s]}}(\widehat{\psi},\widehat{\varphi}) = \ell_{H z_{[s]}}(\widetilde{\psi},\widetilde{\varphi})$ for all $z_{[s]} \in \CC^s$.
\end{lem}
\begin{proof}
By the definition of normalization, we have
$$
\widehat{\psi}(z) = \psi(\widehat{U}^*z), \widehat{\varphi}(z) = \widetilde{A}z + \widehat{b} \text{ with }\widehat{b} = \widehat{V}^*b.
$$
By \cite[Theorem 2.6.5]{HR-90}, there are $(n-s) \times (n-s)$ unitary matrices $V_0, U_0$ and $n_1 \times n_1$ unitary matrix $H_1$,..., $n_s \times n_s$ unitary matrix $H_d$ such that
$$
\widehat{V} = V(H_1 \oplus ...\oplus H_d \oplus V_0) \text{ and } \widehat{U} = (H_1^* \oplus ...\oplus H_d^* \oplus U_0^*)U,
$$
where $n_i$ with $i = 1, 2,...,d,$ is the multiplicity of the distinct positive singular value $\sigma_i$ of $A$ and $\sigma_1 > \sigma_2 > ... > \sigma_d$. In this case, $n_1 + ... + n_d = \textnormal{rank}A = s$.

Putting $H = H_1 \oplus ...\oplus H_d$, we get that $H$ is an $s \times s$ unitary matrix and for every $z \in \CC^n$,
\begin{align*}
\|\widehat{\psi}(z_{[s]},\cdot)\|_{n-s,q} &= \left( \left( \dfrac{q}{2\pi}\right)^{n-s} \int_{\CC^{n-s}} |\psi(\widehat{U}^*z)|^q e^{-\frac{q\left|z'_{[s]}\right|^2}{2}} dA(z'_{[s]})\right)^{\frac{1}{q}} \\
& = \left( \left( \dfrac{q}{2\pi}\right)^{n-s} \int_{\CC^{n-s}} |\widetilde{\psi}(Hz_{[s]}, U_0 z'_{[s]})|^q e^{-\frac{q\left|U_0 z'_{[s]}\right|^2}{2}} dA(z'_{[s]})\right)^{\frac{1}{q}} \\
& = \left( \left( \dfrac{q}{2\pi}\right)^{n-s} \int_{\CC^{n-s}} |\widetilde{\psi}(Hz_{[s]}, z'_{[s]})|^q e^{-\frac{q\left|z'_{[s]}\right|^2}{2}} \right)^{\frac{1}{q}} \\
& = \|\widetilde{\psi}(H z_{[s]},\cdot)\|_{n-s,q}, \text{ since } U_0 \text{ is unitary}.
\end{align*}
Moreover, since $V_0, H$ are unitary and $\widetilde{a}_{s+1,s+1} = ... = \widetilde{a}_{n,n} = 0$, for each $z \in \CC^n$ we have
\begin{align*}
|\widehat{\varphi}(z)|^2 - |z_{[s]}|^2 & = \left| \widehat{V}^* A \widehat{U}^* z + \widehat{V}^*b\right|^2 - |z_{[s]}|^2 \\
& = \left| (H^* \oplus V_0^*)V^* A U^* (H \oplus U_0) z + (H^* \oplus V_0^*)V^*b\right|^2 - |z_{[s]}|^2 \\
& = \left| \widetilde{A} (H \oplus U_0) z + \widetilde{b}\right|^2 - |z_{[s]}|^2  = \left| \widetilde{A} (Hz_{[s]}, U_0z'_{[s]}) + \widetilde{b}\right|^2 - |z_{[s]}|^2 \\
& = \left| \widetilde{A} (Hz_{[s]}, 0'_{[s]}) + \widetilde{b}\right|^2 - |z_{[s]}|^2  = \left|\widetilde{\varphi}(Hz_{[s]},0'_{[s]})\right|^2 - \left|H z_{[s]}\right|^2.
\end{align*}
Consequently, $\ell_{ z_{[s]}}(\widehat{\psi},\widehat{\varphi}) = \ell_{Hz_{[s]}}(\widetilde{\psi},\widetilde{\varphi})$ for all $z \in \CC^n$.
\end{proof}

\begin{rem}
Lemma \ref{lem-UV} guarantees that the criteria for the boundedness and compactness in Theorems \ref{thm-bounded}, \ref{thm-compact}, \ref{thm-bd} and the estimates for essential norm in Theorem \ref{thm-essnorm} of $W_{\psi,\varphi}$ induced by a pair $(\psi,\varphi) \in \calW_q$ in terms of $\ell_{z_{[s]}}(\widetilde{\psi},\widetilde{\varphi})$ are the same for every normalization $(\widetilde{\psi},\widetilde{\varphi})$ (in details, see Remark \ref{rem-incl}).

Moreover, for a pair $(\psi,\varphi) \in \mathcal V_{q,s}$ without loss of generality we can assume that $(\widetilde{\psi},\widetilde{\varphi}) = (\psi,\varphi)$.
\end{rem}

We separate two cases $p \leq q$ and $q < p$, which give different criteria. For an $n \times n$ diagonal matrix $A$ and $0 < s < n$ let us denote by $A_{[s]}$ the submatrix of $A$ with the diagonal entries $a_{ii}, i = 1,...,s$.

\subsection{The case $0 < p \leq q < \infty$}

Firstly we have the following criterion for the boundedness.

\begin{thm}\label{thm-bounded}
Let $0 < p \leq q < \infty$ and $(\psi, \varphi)$ be a pair in $\calW_q$ with $\varphi(z) = Az + b$ and $\text{rank} A = s$. Then the operator $W_{\psi,\varphi}: \calF^p(\CC^n) \to \calF^q(\CC^n)$ is bounded if and only if $\ell(\widetilde{\psi}, \widetilde{\varphi}) < \infty$, where $(\widetilde{\psi}, \widetilde{\varphi})$ is the normalization of $(\psi, \varphi)$ with respect to the singular value decomposition $A = V\widetilde{A}U$.
Moreover,
$$
\ell(\widetilde{\psi}, \widetilde{\varphi}) \leq \|W_{\psi, \varphi}\| \leq |\textnormal{det}\widetilde{A}_{[s]}|^{-\frac{2}{q}} \left( \dfrac{q}{p} \right)^{\frac{n}{q}} \ell(\widetilde{\psi}, \widetilde{\varphi}).
$$
\end{thm}
\begin{proof}
Note that it suffices to prove the theorem for the case when $(\psi, \varphi) \in \mathcal V_{q,s}$. Indeed, if so, then applying the result to $(\widetilde{\psi}, \widetilde{\varphi})$ and using Proposition \ref{prop-equiv}, we can get the assertion.

We prove the theorem for the operators $W_{\psi, \varphi}$ induced by $(\psi, \varphi) \in \mathcal V_{q,s}$.
In this case, $(\widetilde{\psi}, \widetilde{\varphi})= (\psi, \varphi)$.

\textbf{Necessity.} Suppose that $W_{\psi, \varphi}$ is bounded from $\calF^p(\CC^n)$ into $\calF^q(\CC^n)$. By Lemma \ref{lem-F} and the fact that $\|k_w\|_{n,p} = 1$ for every $w \in \CC^n$, we have
\begin{align*}
\|W_{\psi, \varphi}\| &\geq \|W_{\psi, \varphi}k_w\|_{n, q} = \|\psi(z)e^{\langle Az + b, w  \rangle - \frac{|w|^2}{2}}\|_{n, q}\\
& \geq e^{-\frac{\left|z_{[s]}\right|^2}{2}} \|\psi(z_{[s]},\cdot) e^{\langle Az + b, w  \rangle - \frac{|w|^2}{2}}\|_{n-s,q} \\
& = \left| e^{\langle Az + b, w  \rangle - \frac{|w|^2+\left|z_{[s]}\right|^2}{2}}\right| \|\psi(z_{[s]},\cdot) \|_{n-s,q}, \text{ for all } w, z \in \CC^n.
\end{align*}
In particular, with $w = Az + b$, the last inequality gives
\begin{equation}\label{eq-est}
\|W_{\psi, \varphi}\| \geq \|W_{\psi, \varphi}k_{Az + b}\|_{n, q} \geq  e^{ \frac{|A z + b|^2 - \left|z_{[s]}\right|^2}{2}} \|\psi(z_{[s]},\cdot)\|_{n-s, q} = \ell_{z_{[s]}}(\psi, \varphi),
\end{equation}
for all $z_{[s]} \in \CC^{s}$, and so,
$\ell(\psi, \varphi) \leq \|W_{\psi, \varphi}\| < \infty$.

\textbf{Sufficiency.} Suppose that $\ell(\psi,\varphi) < \infty$. Then, for each $f \in \calF^p(\CC^n)$, by \eqref{eq-normf} and Lemmas \ref{lem-F} and \ref{lem-Fpq}, we have
\begin{align*}
&\|W_{\psi,\varphi}f\|^q_{n, q}\\
 =&  \left( \dfrac{q}{2\pi} \right)^s\int_{\CC^s}  |f(Az+b)|^q e^{-\frac{q\left|z_{[s]}\right|^2}{2}} \| \psi(z_{[s]},\cdot) \|^q_{n-s, q}dA(z_{[s]})\\
\leq & \left( \dfrac{q}{2\pi} \right)^s \ell^q(\psi,\varphi) \int_{\CC^s}  |f(Az+b)|^q e^{-\frac{q|Az+b|^2}{2}}dA(z_{[s]}) \\
 = & \left( \dfrac{q}{2\pi} \right)^s \ell^q(\psi,\varphi) \int_{\CC^s} \left|f(A_{[s]}z_{[s]}+b_{[s]},b'_{[s]})\right|^q e^{-\frac{q\left|\left(A_{[s]} z_{[s]}+b_{[s]},b'_{[s]}\right)\right|^2}{2}}dA(z_{[s]}) \\
 \leq & \left( \dfrac{q}{2\pi} \right)^s \ell^q(\psi,\varphi) |\textnormal{det}A_{[s]}|^{-2} \int_{\CC^s}  \left|f(\zeta_{[s]},b'_{[s]})\right|^q e^{-\frac{q\left|\left(\zeta_{[s]},b'_{[s]}\right)\right|^2}{2}}dA(\zeta_{[s]}) \\
\end{align*}
\begin{align*}
 = &\; \ell^q(\psi,\varphi) |\textnormal{det}A_{[s]}|^{-2} \|f(\cdot,b'_{[s]})\|^q_{s, q} e^{-\frac{q\left|b'_{[s]}\right|^2}{2}}
 \leq \ell^q(\psi,\varphi) |\textnormal{det}A_{[s]}|^{-2} \|f\|^q_{n, q} \\
 \leq & \; \ell^q(\psi,\varphi) |\textnormal{det}A_{[s]}|^{-2} \left( \dfrac{q}{p} \right)^n \|f\|^q_{n, p}.
\end{align*}
Hence, $W_{\psi,\varphi}: \calF^p(\CC^n) \to \calF^q(\CC^n) $ is bounded and
$$
\|W_{\psi, \varphi}\| \leq |\textnormal{det}A_{[s]}|^{-\frac{2}{q}} \left( \dfrac{q}{p} \right)^{\frac{n}{q}} \ell(\psi,\varphi).
$$
The assertion is proved for $(\psi,\varphi) \in \mathcal V_{q,s}$.

\end{proof}

Next we have the following criterion for the compactness of weighted composition operators $W_{\psi, \varphi}$.

\begin{thm}\label{thm-compact}
Let $ 0 < p \leq q < \infty$ and $(\psi, \varphi)$ be a pair in $\calW_q$ with $\varphi(z) = Az + b$ and $\text{rank}A = s$.
Then the operator $W_{\psi,\varphi}: \calF^p(\CC^n) \to \calF^q(\CC^n)$ is compact if and only if
$$
\lim_{z_{[s]} \to \infty} \ell_{z_{[s]}}(\widetilde{\psi},\widetilde{\varphi}) = 0,
$$
where, as above, $(\widetilde{\psi}, \widetilde{\varphi})$ is the normalization of $(\psi, \varphi)$.
\end{thm}

\begin{proof}
As in Theorem \ref{thm-bounded}, it suffices to prove the theorem for the operator $W_{\psi,\varphi}$ induced by $(\psi, \varphi)$ in $\mathcal V_{q, s}$, and then using Proposition \ref{prop-equiv} to complete the proof.

Remind that for $(\psi, \varphi)$ in $\mathcal V_{q, s}$, $(\widetilde{\psi}, \widetilde{\varphi})= (\psi, \varphi)$.

\textbf{Necessity.} Suppose that $W_{\psi,\varphi}$ is compact from $\calF^p(\CC^n)$ into $\calF^q(\CC^n)$. Since $\|k_w\|_{n, p} = 1$ for all $w \in \CC^n$ and $k_w \to 0$ in $\calO(\CC^n)$ as $w \to \infty$, by Lemma \ref{lem-com},
$  \|W_{\psi,\varphi}k_w\|_{n, q} \rightarrow 0$ as $w \to \infty$.

From this, \eqref{eq-est}, and the fact that $\varphi(z) = Az + b \to \infty$ as $z_{[s]} \to \infty$ in $\CC^s$, it follows that
$$
\ell_{z_{[s]}}(\psi, \varphi) \leq \|W_{\psi,\varphi}k_{Az+b}\| \to 0, \text{  as } z_{[s]} \to \infty.
$$

\textbf{Sufficiency.} Suppose that
$$
\lim_{z_{[s]} \to \infty}\ell_{z_{[s]}}(\psi,\varphi) = 0.
$$
Then $\ell(\psi,\varphi) < \infty$, and by Theomrem \ref{thm-bounded}, $W_{\psi, \varphi}$ is bounded from $\calF^p(\CC^n)$ into $\calF^q(\CC^n)$.

Let $(f_j)_j$ be a bounded sequence in $\calF^p(\CC^n)$ converging to $0$ in $\calO(\CC^n)$. By \eqref{eq-normf} and Lemmas \ref{lem-F} and \ref{lem-Fpq}, for every $R>0$ and $j \in \NN$,
\begin{align*}
&\|W_{\psi,\varphi}f_j\|^q_{n, q} = \left( \dfrac{q}{2\pi} \right)^s \int_{\CC^s}  |f_j(Az+b)|^q e^{-\frac{q\left|z_{[s]}\right|^2}{2}} \|\psi(z_{[s]},\cdot)\|^q_{n-s, q}dA(z_{[s]})
\end{align*}
\begin{align*}
=& \left( \dfrac{q}{2\pi} \right)^s \int_{|z_{[s]}|\leq R} |f_j(Az+b)|^q e^{-\frac{q\left|z_{[s]}\right|^2}{2}} \|\psi(z_{[s]},\cdot)\|^q_{n-s, q}dA(z_{[s]}) \\
+ & \left( \dfrac{q}{2\pi} \right)^s \int_{|z_{[s]}| > R} |f_j(Az+b)|^q e^{-\frac{q\left|z_{[s]}\right|^2}{2}} \|\psi(z_{[s]},\cdot)\|^q_{n-s, q}dA(z_{[s]}) \\
\leq &  \left( \dfrac{q}{2\pi} \right)^s \max_{|z_{[s]}|\leq R} |f_j(Az+b)|^q \int_{|z_{[s]}| \leq R} e^{-\frac{q\left|z_{[s]}\right|^2}{2}} \|\psi(z_{[s]},\cdot)\|^q_{n-s, q} dA(z_{[s]})  \\
+ & \left( \dfrac{q}{2\pi} \right)^s \max_{|z_{[s]}| > R} \ell^q_{z_{[s]}}(\psi,\varphi) \int_{|z_{[s]}| > R}|f_j(Az+b)|^q e^{-\frac{q|Az+b|^2}{2}}dA(z_{[s]}) \\
\leq & \; \|\psi\|_{n, q}^q \max_{|z_{[s]}|\leq R} |f_j(Az+b)|^q \\
+ & \left( \dfrac{q}{2\pi} \right)^s |\textnormal{det}A_{[s]}|^{-2} \max_{|z_{[s]}| > R} \ell^q_{z_{[s]}}(\psi,\varphi)  \int_{\CC^s}\left|f_j(\zeta_{[s]},b'_{[s]})\right|^q e^{-\frac{q\left|\left(\zeta_{[s]},b'_{[s]}\right)\right|^2}{2}}dA(\zeta_{[s]}) \\
= & \; \|\psi\|_{n, q}^q \max_{|z_{[s]}|\leq R} |f_j(Az+b)|^q \\
+& \; |\textnormal{det}A_{[s]}|^{-2} \|f_j(\cdot,b'_{[s]})\|^q_{s, q} e^{-\frac{q\left|b'_{[s]}\right|^2}{2}} \max_{|z_{[s]}| > R} \ell^q_{z_{[s]}}(\psi,\varphi)\\
\leq & \; \|\psi\|_{n, q}^q \max_{|z_{[s]}|\leq R} |f_j(Az+b)|^q + |\textnormal{det}A_{[s]}|^{-2} \|f_j\|^q_{n, q} \max_{|z_{[s]}| > R} \ell^q_{z_{[s]}}(\psi,\varphi) \\
 \leq & \; \|\psi\|_{n, q}^q \max_{|z_{[s]}|\leq R} |f_j(Az+b)|^q + |\textnormal{det}A_{[s]}|^{-2} \left(\dfrac{q}{p}\right)^n \|f_j\|^q_{n, p} \max_{|z_{[s]}| > R} \ell^q_{z_{[s]}}(\psi,\varphi) \\
\leq & \; \|\psi\|_{n,q}^q \max_{|z_{[s]}|\leq R} |f_j(Az+b)|^q + |\textnormal{det}A_{[s]}|^{-2} \left(\dfrac{q}{p}\right)^n M^q \max_{|z_{[s]}| > R} \ell^q_{z_{[s]}}(\psi,\varphi),
\end{align*}
where
$$
M = \sup_{ j \geq 1} \|f_j\|_{n, p} < \infty.
$$
In the last inequality, letting $j \to \infty$, and then $R \to \infty$, we get
$$
\lim_{j \to \infty} \|W_{\psi,\varphi}f_j\|_{n, q} =0.
$$
By Lemma \ref{lem-com}, $W_{\psi,\varphi}: \calF^p(\CC^n) \to \calF^q(\CC^n)$ is compact.
\end{proof}

From Theorems \ref{thm-bounded}, \ref{thm-compact} and Lemma \ref{lem-as}, we get immediately the following result for the case when $A$ is invertible.

\begin{cor}\label{cor-iv}
Let $0 < p \leq q < \infty$ and $(\psi,\varphi)$ be a pair in $\calW_q$ with $\varphi(z) = Az + b$ and $A$ is invertible. Then the following statements are true:
\begin{itemize}
\item[(a)] The operator $W_{\psi,\varphi}: \calF^p(\CC^n) \to \calF^q(\CC^n)$ is bounded if and only if $m(\psi,\varphi) < \infty$ and
$$
m(\psi, \varphi) \leq \|W_{\psi, \varphi}\| \leq |\textnormal{det}A|^{-\frac{2}{q}} \left( \dfrac{q}{p} \right)^{\frac{n}{q}} m(\psi,\varphi).
$$

\item[(b)] The operator $W_{\psi,\varphi}: \calF^p(\CC^n) \to \calF^q(\CC^n)$ is compact if and only if
$$
\lim_{z \to \infty} m_z(\psi,\varphi) = 0.
$$
\end{itemize}
\end{cor}

In particular, when $\psi \equiv { \rm const }$ on $\CC^n$ we obtain the following result for composition operators.

\begin{cor}\label{cor-co}
Let $0< p \leq q < \infty$ and $\varphi: \CC^n \to \CC^n$ a holomorphic mapping. The following statements are true:
\begin{itemize}
\item[(a)] The operator $C_{\varphi}: \calF^p(\CC^n) \to \calF^q(\CC^n)$ is bounded if and only if $\varphi(z) = Az + b$, where $A$ is an $n \times n$ matrix and  $b$ is an $n \times 1$ vector such that $\|A\| \leq 1$ and $\langle A\zeta, b \rangle =0$ for every $\zeta$ in $\CC^n$ with $|A\zeta| = |\zeta|$.\
\item[(b)] The operator $C_{\varphi}: \calF^p(\CC^n) \to \calF^q(\CC^n)$ is compact if and only if $\varphi(z) = Az + b$, where $A$ is an $n \times n$ matrix and  $b$ is an $n \times 1$ vector such that $\|A\| < 1$.
\end{itemize}
\end{cor}
\begin{proof}
(a) The necessity follows directly from Proposition \ref{prop-nec} and Corollary \ref{cor-varphi}. We prove the sufficiency. Suppose that $\varphi(z) = Az + b$ with $\textnormal{rank}A = s$ and $\widetilde{\varphi}(z) = \widetilde{A} z + \widetilde{b}$, where the singular value decomposition of $A$ is $V \widetilde{A} U$ and $\widetilde{b} = V^*b$.

Since $\|A\| \leq 1$, we have $\|\widetilde{A}\| \leq 1$, and hence,
$$
1 \geq \widetilde{a}_{11} \geq \widetilde{a}_{22} \geq ... \geq \widetilde{a}_{ss} \geq \widetilde{a}_{s+1,s+1} = ... = \widetilde{a}_{nn}=0.
$$
Put
$ j = \max\{i: \widetilde{a}_{ii} = 1\}$. By \cite[Lemma 1]{CMS-03}, the first $j$ coordinates of $\widetilde{b}$ are $0$.
Then, for every $z \in \CC^n$,
\begin{align}\label{eq-var}
\ell_{z_{[s]}}(1, \widetilde{\varphi}) &= e^{\frac{|\widetilde{\varphi}(z)|^2 - \left|z_{[s]}\right|^2}{2}} \|1\|_{n-s, q} \\
& = \exp{\frac{\sum_{i=j+1}^s (|\widetilde{a}_{ii}z_i+\widetilde{b}_i|^2- |z_i|^2) + \sum_{i=s+1}^n |\widetilde{b}_i|^2  }{2}}. \nonumber
\end{align}
Since $\widetilde{a}_{ii} < 1$ for all $j+1 \leq i \leq s$, we have $\ell(1, \widetilde{\varphi}) < \infty$. By Theorem \ref{thm-bounded}, the operator $C_{\widetilde{\varphi}}: \calF^p(\CC^n) \to \calF^q(\CC^n)$ is bounded, and hence, by Proposition \ref{prop-equiv}, so is $C_{\varphi}$.

(b) By Proposition \ref{prop-equiv}, Theorem \ref{thm-compact}, and \eqref{eq-var}, the operators $C_{\varphi}$ and $C_{\widetilde{\varphi}}: \calF^p(\CC^n) \to \calF^q(\CC^n)$ are compact if and only if
$$
\lim_{z_{[s]} \to \infty} \ell_{z_{[s]}}(1, \widetilde{\varphi}) = 0,
$$
i. e.
$$
\lim_{z_{[s]} \to \infty} \left( \sum_{i=j+1}^s (|\widetilde{a}_{ii}z_i+\widetilde{b}_i|^2- |z_i|^2) + \sum_{i=s+1}^n |\widetilde{b}_i|^2 \right) = -\infty,
$$
which is equivalent to that $j = 0$, that is, $\widetilde{a}_{ii} < 1$ for all $1 \leq i \leq s$, and hence, $\|\widetilde{A}\| < 1$ and $\|A\| < 1$.
\end{proof}

\subsection{The case $0 < q < p < \infty$}

For each pair $(\psi, \varphi)$ in $\mathcal V_{q,s}$, we define the following positive pull-back measure $\mu_{\psi, \varphi, q}$ on $\CC^s$
$$
\mu_{\psi, \varphi, q}(E) = \left( \dfrac{q}{2\pi}\right)^{s} \int_{\varphi_{[s]}^{-1}(E)} \|\psi(z_{[s]},\cdot)\|_{n-s,q}^q e^{-\frac{q\left|z_{[s]}\right|^2}{2}}dA(z_{[s]}),
$$
for every Borel subset $E$ of $\CC^s$, where $\varphi_{[s]}(z_{[s]}) = A_{[s]}z_{[s]} + b_{[s]}, \ z_{[s]} \in \CC^s$.

We recall that for $p, q \in (0, \infty)$ a positive Borel measure $\mu$ on $\CC^s$ is called a \textit{$(p,q)$-Fock Carleson measure}, if the embedding operator $i: \calF^p(\CC^s) \to L^q(\CC^s,d\mu)$ is bounded, i.e. if there exists a constant $C>0$ such that for every $f \in \calF^p(\CC^s)$,
$$
\left( \int_{\CC^s} |f(z)|^q e^{-\frac{q|z|^2}{2}}d\mu(z) \right)^{\frac{1}{q}} \leq C \|f\|_{s, p}.
$$
We write $\|\mu\|$ for the operator norm of $i$ from $\calF^p(\CC^s)$ to $L^q(\CC^s,d\mu)$ and refer the reader to \cite[Section 3]{HL-11} for more information about $(p,q)$-Fock Carleson measure.

\begin{thm}\label{thm-bd}
Let $0 < q < p < \infty$ and $(\psi, \varphi)$ be a pair in $\calW_q$ with $\varphi(z) = Az + b$ and $\text{rank} A = s$. Then the following assertions are equivalent:
\begin{itemize}
\item[(i)] $W_{\psi,\varphi}: \calF^p(\CC^n) \to \calF^q(\CC^n)$ is bounded;
\item[(ii)] $W_{\psi,\varphi}: \calF^p(\CC^n) \to \calF^q(\CC^n)$ is compact;
\item[(iii)] $\ell_{z_{[s]}}(\widetilde{\psi},\widetilde{\varphi}) \in L^{\frac{pq}{p-q}}(\CC^s,dA)$,
\end{itemize}
where (as in Theorem \ref{thm-bounded}) $(\widetilde{\psi}, \widetilde{\varphi})$ is the normalization of $(\psi, \varphi)$ with respect to the singular value decomposition $A = V\widetilde{A}U$.

In this case, for some positive constant $C$,
\begin{align*}
C^{-1} |\textnormal{det}\widetilde{A}_{[s]}|^{\frac{2(p-q)}{pq}}& e^{-\frac{\left|\widetilde{b}'_{[s]}\right|^2}{2}} \|\ell_{z_{[s]}}(\widetilde{\psi},\widetilde{\varphi})  \|_{L^{\frac{pq}{p-q}}}\\
 & \leq \|W_{\psi,\varphi}\| \leq C |\textnormal{det}\widetilde{A}_{[s]}|^{-\frac{2}{p}} \|\ell_{z_{[s]}}(\widetilde{\psi},\widetilde{\varphi})\|_{L^{\frac{pq}{p-q}}}.
\end{align*}
\end{thm}
\begin{proof} Similarly to the proof of Theorem \ref{thm-bounded}, it suffices to prove the theorem for the operator $W_{\psi,\varphi}$ induced by $(\psi, \varphi)$ in $\mathcal V_{q, s}$, and then using Proposition \ref{prop-equiv} to complete the proof.

Again note that for $(\psi, \varphi)$ in $\mathcal V_{q, s}$, $(\widetilde{\psi}, \widetilde{\varphi}) = (\psi, \varphi)$.

$\bullet$ (ii) $\Longrightarrow$ (i) is obvious.

$\bullet$ (i) $\Longrightarrow$ (iii). Suppose that the operator $W_{\psi,\varphi}: \calF^p(\CC^n) \to \calF^q(\CC^n)$ is bounded. Then by \eqref{eq-normf}, for every $f \in \calF^p(\CC^n)$
\begin{align*}
&\|W_{\psi,\varphi}\| \|f\|_{n,p} \geq  \|W_{\psi,\varphi}f\|_{n, q} \\
 = & \left(\left(\dfrac{q}{2\pi}\right)^s \int_{\CC^s} \|\psi(z_{[s]},\cdot)\|_{n-s,q}^q |f(Az+b)|^q e^{-\frac{q\left|z_{[s]}\right|^2}{2}}dA(z_{[s]}) \right)^{\frac{1}{q}} \\
 = & \left(\left(\dfrac{q}{2\pi}\right)^s \int_{\CC^s} \|\psi(z_{[s]},\cdot)\|_{n-s,q}^q \left|f(A_{[s]}z_{[s]}+b_{[s]}, b'_{[s]})\right|^q e^{-\frac{q\left|z_{[s]}\right|^2}{2}}dA(z_{[s]}) \right)^{\frac{1}{q}}.
\end{align*}
This implies that for every $f \in \calF^p(\CC^s)$, i.e. for every $f \in \calF^p(\CC^n)$ independent on $(z_{s+1},...,z_n)$, we have
\begin{align*}
&\|W_{\psi,\varphi}\| \|f\|_{s,p} = \|W_{\psi,\varphi}\| \|f\|_{n,p} \\
 = & \left(\left(\dfrac{q}{2\pi}\right)^s \int_{\CC^s} \|\psi(z_{[s]},\cdot)\|_{n-s,q}^q \left|f(A_{[s]}z_{[s]}+b_{[s]})\right|^q e^{-\frac{q\left|z_{[s]}\right|^2}{2}}dA(z_{[s]}) \right)^{\frac{1}{q}} \\
 =& \left( \int_{\CC^s} \left|f(\zeta_{[s]})\right|^q d\mu_{\psi,\varphi,q}(\zeta_{[s]}) \right)^{\frac{1}{q}}
  = \left( \int_{\CC^s} |f(\zeta_{[s]})|^q e^{-\frac{q\left|\zeta_{[s]}\right|^2}{2}}d\lambda_{\psi,\varphi,q}(\zeta_{[s]}) \right)^{\frac{1}{q}},
\end{align*}
where $d\lambda_{\psi,\varphi,q}(\zeta_{[s]}) = e^{\frac{q\left|\zeta_{[s]}\right|^2}{2}} d\mu_{\psi,\varphi,q}(\zeta_{[s]})$.
The last inequality means that $\lambda_{\psi,\varphi,q}$ is a $(p, q)$-Fock Carleson measure on $\CC^s$. Then, by \cite[Theorem 3.3]{HL-11}, we have
\begin{equation} \label{eq-lam}
\widetilde{\lambda_{\psi,\varphi,q}}(w_{[s]}) = \int_{\CC^s} \left|k_{w_{[s]}}(z_{[s]})\right|^q e^{-\frac{q\left|z_{[s]}\right|^2}{2}} d\lambda_{\psi,\varphi,q}(z_{[s]}) \in L^{\frac{p}{p-q}}(\CC^s, dA).
\end{equation}
On the other hand, for all $w \in \CC^n$,
\begin{align*}
& \widetilde{\lambda_{\psi,\varphi,q}}(w_{[s]}) \\
= & \int_{\CC^s} \left|k_{w_{[s]}}(z_{[s]})\right|^q e^{-\frac{q\left|z_{[s]}\right|^2}{2}} d\lambda_{\psi,\varphi,q}(z_{[s]})
=  \int_{\CC^s} \left|k_{w_{[s]}}(z_{[s]})\right|^q d\mu_{\psi,\varphi,q}(z_{[s]}) \\
= & \left( \dfrac{q}{2\pi} \right)^s \int_{\CC^s}\|\psi(z_{[s]},\cdot)\|_{n-s,q}^q \left|k_{w_{[s]}}(A_{[s]}z_{[s]} + b_{[s]})\right|^q e^{-\frac{q\left|z_{[s]}\right|^2}{2}}dA(z_{[s]}) \\
= & \left( \dfrac{q}{2\pi} \right)^n \int_{\CC^n}|\psi(z_{[s]},z'_{[s]})|^q \left|k_{w_{[s]}}(A_{[s]}z_{[s]} + b_{[s]})\right|^q e^{-\frac{q|z|^2}{2}}dA(z) \\
= & \left( \dfrac{q}{2\pi} \right)^n \int_{\CC^n}\left|\psi(z) k_{w_{[s]}}(A_{[s]}z_{[s]} + b_{[s]})\right|^q e^{-\frac{q|z|^2}{2}}dA(z) \\
= &\; \|\psi \cdot ( k_{w_{[s]}}\circ \varphi_{[s]} )\|^q_{n, q},
\end{align*}
where, as above, $\varphi_{[s]}(z_{[s]}) = A_{[s]}z_{[s]} + b_{[s]}$.
From this and Lemma \ref{lem-F}, it follows that for all $w, z \in \CC^n$,
\begin{align*}
 \widetilde{\lambda_{\psi,\varphi,q}}(w_{[s]}) & \geq e^{- \frac{q\left|z_{[s]}\right|^2}{2}}  \|\psi(z_{[s]},\cdot) (k_{w_{[s]}}\circ \varphi_{[s]})(z_{[s]})\|^q_{n-s, q}\\
& = \left|e^{\langle A_{[s]}z_{[s]} + b_{[s]}, w_{[s]}\rangle - \frac{\left|z_{[s]} \right|^2+\left|w_{[s]}\right|^2}{2}} \right|^q  \|\psi(z_{[s]},\cdot)\|^q_{n-s, q}.
\end{align*}
In particular, with $w = \varphi(z) = Az + b$, we get
\begin{align*}
& \widetilde{\lambda_{\psi,\varphi,q}}(\varphi_{[s]}(z_{[s]}))  \geq e^{\frac{q\left(\left|\varphi_{[s]}(z_{[s]})\right|^2 - \left|z_{[s]}\right|^2\right)}{2}} \|\psi(z_{[s]},\cdot)\|^q_{n-s, q} \\
& =  e^{-\frac{q\left|b'_{[s]}\right|^2}{2}}e^{\frac{q\left(\left|\varphi(z)\right|^2 - \left|z_{[s]}\right|^2\right)}{2}} \|\psi(z_{[s]},\cdot)\|^q_{n-s, q}
 = e^{-\frac{q\left|b'_{[s]}\right|^2}{2}} \ell^q_{z_{[s]}}(\psi,\varphi),
\end{align*}
for all $z \in \CC^n$. Combining this and \eqref{eq-lam} yields
\begin{align*}
& e^{-\frac{qp\left|b'_{[s]}\right|^2}{2(p-q)}} \int_{\CC^s} (\ell_{z_{[s]}}(\psi,\varphi))^{\frac{pq}{p-q}}dA(z_{[s]})
\leq \int_{\CC^s} \left(\widetilde{\lambda_{\psi,\varphi,q}}(\varphi_{[s]}(z_{[s]}))\right)^{\frac{p}{p-q}}dA(z_{[s]}) \\
& = |\textnormal{det}A_{[s]}|^{-2} \int_{\CC^s} \left( \widetilde{\lambda_{\psi,\varphi,q}}(\zeta_{[s]})\right)^{\frac{p}{p-q}}dA(\zeta_{[s]})
 =  |\textnormal{det}A_{[s]}|^{-2} \|\widetilde{\lambda_{\psi,\varphi,q}}\|_{L^{\frac{p}{p-q}}}^{\frac{p}{p-q}} < \infty.
\end{align*}
Therefore, $\ell_{z_{[s]}}(\psi,\varphi) \in L^{\frac{pq}{p-q}}(C^s,dA)$ and
$$
|\textnormal{det}A_{[s]}|^{\frac{2(p-q)}{p}} e^{-\frac{q\left|b'_{[s]}\right|^2}{2}} \|\ell_{z_{[s]}}(\psi,\varphi)\|^q_{L^{\frac{pq}{p-q}}} \leq \|\widetilde{\lambda_{\psi,\varphi,q}}\|_{L^{\frac{p}{p-q}}}.
$$
From this and \cite[Theorem 3.3]{HL-11}, we see that for some constant $C_1 > 0$,
\begin{align*}
\|W_{\psi,\varphi}\|^q &= \|\lambda_{\psi,\varphi,q}\|^q \geq C_1 \|\widetilde{\lambda_{\psi,\varphi,q}}\|_{L^{\frac{p}{p-q}}} \\
& \geq C_1 |\textnormal{det}A_{[s]}|^{\frac{2(p-q)}{p}} e^{-\frac{q\left|b'_{[s]}\right|^2}{2}} \|\ell_{z_{[s]}}(\psi,\varphi)\|^q_{L^{\frac{pq}{p-q}}},
\end{align*}
which gives
\begin{equation}\label{eq-norm1}
\|W_{\psi,\varphi}\|\geq C_1^{\frac{1}{q}} |\textnormal{det}A_{[s]}|^{\frac{2(p-q)}{pq}} e^{-\frac{\left|b'_{[s]}\right|^2}{2}} \|\ell_{z_{[s]}}(\psi,\varphi)\|_{L^{\frac{pq}{p-q}}}.
\end{equation}

$\bullet$ (iii) $\Longrightarrow$ (ii). For each function $f \in \calF^p(\CC^n)$, using \eqref{eq-normf}, H\"older's inequality and Lemma \ref{lem-F}, we have
\begin{align*}
& \|W_{\psi, \varphi}f\|^q_{n, q} = \left(\dfrac{q}{2\pi}\right)^s \int_{\CC^s} \|\psi(z_{[s]},\cdot)\|_{n-s,q}^q \left|f(Az+b)\right|^q e^{-\frac{q\left|z_{[s]}\right|^2}{2}}dA(z_{[s]})\\
= & \left(\dfrac{q}{2\pi}\right)^s \int_{\CC^s} \ell^q_{z_{[s]}}(\psi, \varphi) \left|f(Az+b)\right|^q e^{-\frac{q\left|Az+b\right|^2}{2}}dA(z_{[s]})\\
 \leq & \left(\dfrac{q}{2\pi}\right)^s \left(\int_{\CC^s} \ell_{z_{[s]}}^{\frac{pq}{p-q}}(\psi, \varphi) dA(z_{[s]})\right)^{\frac{p-q}{p}} \\
& \qquad \qquad \times \left( \int_{\CC^s} \left|f(Az+b)\right|^p e^{-\frac{p\left|Az+b\right|^2}{2}}dA(z_{[s]})\right)^{\frac{q}{p}} \\
 =&  \left(\dfrac{q}{2\pi}\right)^s \left(\dfrac{2\pi}{p}\right)^{\frac{sq}{p}} \|\ell_{z_{[s]}}(\psi,\varphi)\|_{L^{\frac{pq}{p-q}}}^q |\textnormal{det}A_{[s]}|^{-\frac{2q}{p}} e^{-\frac{q\left|b'_{[s]}\right|^2}{2}} \|f(\cdot,b'_{[s]})\|^q_{s,p} \\
 \leq & \left(\dfrac{q}{2\pi}\right)^s \left(\dfrac{2\pi}{p}\right)^{\frac{sq}{p}} \|\ell_{z_{[s]}}(\psi,\varphi)\|_{L^{\frac{pq}{p-q}}}^q |\textnormal{det}A_{[s]}|^{-\frac{2q}{p}} \|f\|^q_{n,p}.
\end{align*}
This shows that the operator $W_{\psi, \varphi}: \calF^p(\CC^n) \to \calF^q(\CC^n)$ is bounded and
\begin{equation}\label{eq-norm2}
\|W_{\psi,\varphi}\| \leq \left(\dfrac{q}{2\pi}\right)^{\frac{s}{q}} \left(\dfrac{2\pi}{p}\right)^{\frac{s}{p}} |\textnormal{det}A_{[s]}|^{-\frac{2}{p}} \|\ell_{z_{[s]}}(\psi,\varphi)\|_{L^{\frac{pq}{p-q}}}.
\end{equation}

Next, let $(f_j)_j$ be an arbitrary bounded sequence in $\calF^p(\CC^n)$ converging to $0$ in $\calO(\CC^n)$. For each $j \in \NN$ and $R>0$, we have
\begin{align*}
\|W_{\psi,\varphi}f_j\|_{n,q}^q =& \left(\dfrac{q}{2\pi}\right)^s \int_{\CC^s} \ell^q_{z_{[s]}}(\psi, \varphi) |f_j(Az+b)|^q e^{-\frac{q|Az+b|^2}{2}}dA(z_{[s]})\\
=& \left(\dfrac{q}{2\pi}\right)^s \int_{|z_{[s]}|\leq R} \ell^q_{z_{[s]}}(\psi, \varphi) \left|f_j(Az+b)\right|^q e^{-\frac{q|Az+b|^2}{2}}dA(z_{[s]})\\
+& \left(\dfrac{q}{2\pi}\right)^s \int_{|z_{[s]}|>R} \ell^q_{z_{[s]}}(\psi, \varphi) \left|f_j(Az+b)\right|^q e^{-\frac{q|Az+b|^2}{2}}dA(z_{[s]})\\
=&\; \calI(j,R) + \calJ(j,R).
\end{align*}
On one hand, for $\calI(j,R)$, we have
\begin{align*}
\calI(j,R) &\leq \left(\dfrac{q}{2\pi}\right)^s \max_{|z_{[s]}|\leq R} |f_j(Az+b)|^q \int_{|z_{[s]}|\leq R} \ell^q_{z_{[s]}}(\psi, \varphi) e^{-\frac{q|Az+b|^2}{2}}dA(z_{[s]})\\
& = \left(\dfrac{q}{2\pi}\right)^s \max_{|z_{[s]}|\leq R} |f_j(Az+b)|^q \int_{|z_{[s]}|\leq R} \|\psi(z_{[s]},\cdot)\|_{n-s,q}^q e^{-\frac{q \left|z_{[s]}\right|^2}{2}}dA(z_{[s]})\\
& \leq \|\psi\|^q_{n,q} \max_{|z_{[s]}|\leq R} |f_j(Az+b)|^q.
\end{align*}
On the other hand, for $\calJ(j,R)$, again using H\"older's inequality and Lemma \ref{lem-F}, we get
\begin{align*}
&\calJ(j,R) = \left(\dfrac{q}{2\pi}\right)^s \int_{|z_{[s]}|>R} \ell^q_{z_{[s]}}(\psi, \varphi) |f_j(Az+b)|^q e^{-\frac{q|Az+b|^2}{2}}dA(z_{[s]})\\
 \leq & \left(\dfrac{q}{2\pi}\right)^s \left(\int_{|z_{[s]}|>R} \ell^{\frac{pq}{p-q}}_{z_{[s]}}(\psi, \varphi) dA(z_{[s]})\right)^{\frac{p-q}{p}} \\
& \qquad \qquad \qquad \times \left( \int_{|z_{[s]}|>R} |f_j(Az+b)|^p e^{-\frac{p|Az+b|^2}{2}}dA(z_{[s]})\right)^{\frac{q}{p}} \\
\leq & \left(\dfrac{q}{2\pi}\right)^s \left(\dfrac{2\pi}{p}\right)^{\frac{sq}{p}} |\textnormal{det}A_{[s]}|^{-\frac{2q}{p}} e^{-\frac{q\left|b'_{[s]}\right|^2}{2}} \|f_j(\cdot,b'_{[s]})\|^q_{s,p}\\
& \qquad \qquad \qquad \times \left(\int_{|z_{[s]}|>R} \ell^{\frac{pq}{p-q}}_{z_{[s]}}(\psi, \varphi) dA(z_{[s]})\right)^{\frac{p-q}{p}}\\
 \leq & \left(\dfrac{q}{2\pi}\right)^s \left(\dfrac{2\pi}{p}\right)^{\frac{sq}{p}} |\textnormal{det}A_{[s]}|^{-\frac{2q}{p}} \|f_j\|^q_{n,p} \left(\int_{|z_{[s]}|>R} \ell^{\frac{pq}{p-q}}_{z_{[s]}}(\psi, \varphi) dA(z_{[s]})\right)^{\frac{p-q}{p}}\\
 \leq & M^q \left(\int_{|z_{[s]}|>R} \ell^{\frac{pq}{p-q}}_{z_{[s]}}(\psi, \varphi) dA(z_{[s]})\right)^{\frac{p-q}{p}},
\end{align*}
where
$$
M^q = \left(\dfrac{q}{2\pi}\right)^s \left(\dfrac{2\pi}{p}\right)^{\frac{sq}{p}} |\textnormal{det}A_{[s]}|^{-\frac{2q}{p}} \sup_{j} \|f_j\|^q_{n,p} < \infty.
$$
Therefore, for every $R>0$ we get
\begin{align*}
\limsup_{j\to \infty} \|W_{\psi,\varphi}f_j\|_{n,q}^q & \leq \limsup_{j \to \infty} (\calI(j,R) + \calJ(j,R))\\
& \leq M^q \left( \int_{|z_{[s]}|>R} \ell^{\frac{pq}{p-q}}_{z_{[s]}}(\psi, \varphi) dA(z_{[s]})\right)^{\frac{p-q}{p}}.
\end{align*}
Since $\ell_{z_{[s]}}(\psi, \varphi) \in L^{\frac{pq}{p-q}}(\CC^s,dA)$, letting $R \to \infty$ in the last inequality, we get that $W_{\psi,\varphi}f_j$ converges to $0$ in $\calF^q(\CC^n)$ as $j \to \infty$.

Consequently, by Lemma \ref{lem-com}, the operator $W_{\psi,\varphi}: \calF^p(\CC^n) \to \calF^q(\CC^n)$ is compact.

Finally, the desired estimates for $\|W_{\psi,\varphi}\|$ follow from \eqref{eq-norm1} and \eqref{eq-norm2}.
\end{proof}

From Theorems \ref{thm-bd} and Lemma \ref{lem-as}, we obtain immediately the following result for the case when $A$ is invertible.

\begin{cor}\label{cor-iv-1}
Let $0 < q < p < \infty$ and $(\psi,\varphi)$ be a pair in $\calW_q$ with $\varphi(z) = Az + b$ and $A$ is invertible. Then the following assertions are equivalent:
\begin{itemize}
\item[(i)] $W_{\psi,\varphi}: \calF^p(\CC^n) \to \calF^q(\CC^n)$ is bounded;
\item[(ii)] $W_{\psi,\varphi}: \calF^p(\CC^n) \to \calF^q(\CC^n)$ is compact;
\item[(iii)] $m_z(\psi,\varphi) \in L^{\frac{pq}{p-q}}(\CC^n, dA)$.
\end{itemize}
In this case, for some positive constant $C$,
\begin{align*}
C^{-1} |\textnormal{det}A|^{\frac{2(p-q)}{pq}}& \|m_{z}(\psi, \varphi)  \|_{L^{\frac{pq}{p-q}}}\\
 & \leq \|W_{\psi,\varphi}\| \leq C |\textnormal{det}A|^{-\frac{2}{p}} \|m_{z}(\psi,\varphi)\|_{L^{\frac{pq}{p-q}}}.
\end{align*}
\end{cor}

In particular, when $\psi \equiv { \rm const }$ on $\CC^n$ we get the following result for composition operators.

\begin{cor}\label{cor-co-1}
Let $0< q < p < \infty$ and $\varphi: \CC^n \to \CC^n$ a holomorphic mapping. The following assertions are equivalent:
\begin{itemize}
\item[(i)] $C_{\varphi}: \calF^p(\CC^n) \to \calF^q(\CC^n)$ is bounded;
\item[(ii)] $C_{\varphi}: \calF^p(\CC^n) \to \calF^q(\CC^n)$ is compact;
\item[(iii)] $\varphi(z) = Az + b$, where $A$ is an $n \times n$ matrix with $\|A\| < 1$ and $b$ is an $n \times 1$ vector.
\end{itemize}
\end{cor}
\begin{proof}
$\bullet$ (i) $\Longleftrightarrow$ (ii) follows by Theorem \ref{thm-bd}.

$\bullet$ (ii) $\Longleftrightarrow$ (iii).
In view of Proposition \ref{prop-nec} and Corollary \ref{cor-varphi}, we may assume that $\varphi(z) = Az + b$ and $\widetilde{\varphi}(z) = \widetilde{A} z + \widetilde{b}$ as in Corollary \ref{cor-co}.

By Theorem \ref{thm-bd}, $C_{\varphi}$, and hence, $C_{\widetilde{\varphi}}: \calF^p(\CC^n) \to \calF^q(\CC^n)$ is compact if and only if $\ell_{z_{[s]}}(1,\widetilde{\varphi}) \in L^{\frac{pq}{p-q}}(\CC^s,dA)$.

Moreover, by \eqref{eq-var}, for every $z \in \CC^n$,
$$
\ell_{z_{[s]}}(1, \widetilde{\varphi})= \exp{\frac{\sum_{i=j+1}^s (|\widetilde{a}_{ii}z_i+\widetilde{b}_i|^2- |z_i|^2) + \sum_{i=s+1}^n |\widetilde{b}_i|^2  }{2}},
$$
with, as in Corollary \ref{cor-co}, $j = \max\{i: \sigma_{ii} = 1\}$.

It implies that the fact $\ell_{z_{[s]}}(1,\widetilde{\varphi}) \in L^{\frac{pq}{p-q}}(\CC^s,dA)$
is equivalent to that $j=0$, i. e., $\widetilde{a}_{ii} < 1$ for all $1 \leq i \leq s$. That is, $\|\widetilde{A}\| < 1$, and hence, $\|A\| < 1$.
\end{proof}

Now we discuss several particular cases of the main results above.

\begin{rem}
On Fock spaces $\calF^p(\CC)$, i.e. in the case $n=1$, there are only 2 cases of entire functions $\varphi(z) = az + b, a, b \in \CC$.

- Case 1. $a = 0$. Obviously, Proposition \ref{prop-zero} implies the corresponding result in \cite[Corollary~3.2]{TK-17}.

- Case 2. $a \neq 0$. In this case, Corollaries \ref{cor-iv} and \ref{cor-iv-1} yield the corresponding results in \cite[Theorems 3.3, 3,4]{TK-17}.
\end{rem}

\begin{rem}
Corollaries \ref{cor-co} and \ref{cor-co-1} extend the corresponding results for composition operators on Hilbert Fock spaces $\calF^2(\CC^n)$ in \cite[Theorems 1 and 2]{CMS-03} to composition operators acting from a general Fock space into another one.
\end{rem}

\section{Essential norm}

In a general setting, let $X, Y$ be Banach spaces, and $\mathcal{K}(X,Y)$ be the set of all compact operators from $X$ into $Y$. The essential norm of a bounded linear operator $L: X \to Y$, denoted by $\|L\|_e$, is defined as
$$
\|L\|_e=\inf\{\|L-K\|: K\in\mathcal{K}(X,Y)\}.
$$
Clearly, $L$ is compact if and only if $\|L\|_e=0$.

In view of Proposition \ref{prop-zero}, Theorem \ref{thm-bd} and Lemma \ref{lem-com1}, we only study essential norm of $W_{\psi, \varphi}: \mathcal F^p(\CC^n) \to \mathcal F^q(\CC^n)$ when $1< p \leq q < \infty$ and $(\psi, \varphi) \in \calW_q$.

\begin{thm}\label{thm-essnorm}
Let $1 < p \leq q < \infty$ and $W_{\psi,\varphi}: \calF^p(\CC^n) \to \calF^q(\CC^n)$ be a bounded weighted composition operator induced by a pair $(\psi, \varphi) \in \calW_q$ with $\varphi(z) = Az + b$ and $\text{rank} A = s$. Then
$$
\limsup_{z_{[s]}\to\infty}\ell_{z_{[s]}}(\widetilde{\psi}, \widetilde{\varphi}) \leq \|W_{\psi,\varphi}\|_e \leq 2 |\textnormal{det}\widetilde{A}_{[s]}|^{-\frac{2}{q}} \left(\dfrac{q}{p}\right)^{\frac{n}{q}} \limsup_{z_{[s]}\to\infty}\ell_{z_{[s]}}(\widetilde{\psi}, \widetilde{\varphi}),
$$
where (as in Theorem \ref{thm-bounded}) $(\widetilde{\psi}, \widetilde{\varphi})$ is the normalization of $(\psi, \varphi)$ with respect to the singular value decomposition $A = V\widetilde{A}U$.
\end{thm}

\begin{proof}
Since  the operator $W_{\psi,\varphi}: \calF^p(\CC^n) \to \calF^q(\CC^n)$ is bounded, by Theorem \ref{thm-bounded}, $\ell(\widetilde{\psi}, \widetilde{\varphi}) < \infty$. Then $\displaystyle \limsup_{z_{[s]}\to\infty}\ell_{z_{[s]}}(\widetilde{\psi}, \widetilde{\varphi})$ is finite.

Firstly we show that $\|W_{\psi,\varphi}\|_e = \|W_{\widetilde{\psi}, \widetilde{\varphi}}\|_e$. Indeed, for every compact operator $T: \calF^p(\CC^n) \to \calF^q(\CC^n)$ we put
$\widetilde{T}= C_{U^*} T C_{V^*}$. Note that $\widetilde{T}$ is also compact from $\calF^p(\CC^n)$ to $\calF^q(\CC^n)$ and, by the proof of Proposition \ref{prop-equiv}, we have
$$
\|W_{\psi,\varphi} - T\| = \|C_U W_{\widetilde{\psi}, \widetilde{\varphi}}C_V - C_U \widetilde{T} C_V\| \leq  \|W_{\widetilde{\psi}, \widetilde{\varphi}} -  \widetilde{T}\|
$$
and also
$$
\|W_{\widetilde{\psi}, \widetilde{\varphi}} - \widetilde{T}\| = \|C_{U^*} W_{\psi, \varphi}C_{V^*} - C_{U^*} T C_{V^*}\| \leq  \|W_{\psi,\varphi} - T\|.
$$
Then, $\|W_{\widetilde{\psi}, \widetilde{\varphi}} - \widetilde{T}\| = \|W_{\psi,\varphi} - T\|$, which implies that $\|W_{\psi,\varphi}\|_e = \|W_{\widetilde{\psi}, \widetilde{\varphi}}\|_e$.

In view of this, it is enough to prove the theorem for those operators $W_{\psi,\varphi}$ which are induced by $(\psi,\varphi) \in \mathcal V_{q,s}$. In this case, $(\widetilde{\psi}, \widetilde{\varphi}) = (\psi,\varphi)$.

\textbf{Lower estimate.} By contradiction we assume that
$$
\|W_{\psi,\varphi}\|_e < \limsup_{ z_{[s]} \to\infty}\ell_{z_{[s]}}(\psi,\varphi).
$$
Then there exist positive constants $N < M$ and a compact operator $T: \calF^p(\CC^n) \to \calF^q(\CC^n)$ such that
$$
\|W_{\psi,\varphi} - T\| < N < M < \limsup_{z_{[s]}\to\infty}\ell_{z_{[s]}}(\psi,\varphi).
$$
By the definition of $\limsup$, we can find a sequence $(z^j_{[s]})_j$ in $\CC^s$ with $|z^j_{[s]}| \uparrow \infty$ as $j \to \infty$ so that
\begin{equation}\label{eq-in-1}
\lim_{j\to\infty}\ell_{z^j_{[s]}}(\psi,\varphi) = \limsup_{z_{[s]} \to\infty}\ell_{z_{[s]}}(\psi,\varphi) > M.
\end{equation}
On the other hand, for each $j \in \NN$ putting $z^j = (z^j_{[s]}, 0,..., 0) \in \CC^n$ and using \eqref{eq-est}, we have
\begin{align*}
 \|W_{\psi, \varphi} - T\| &\geq \|W_{\psi, \varphi}k_{\varphi(z^j)} - Tk_{\varphi(z^j)}\|_{n, q} \\
& \geq  \|W_{\psi, \varphi}k_{\varphi(z^j)}\|_{n, q} - \|Tk_{\varphi(z^j)}\|_{n, q} \geq \ell_{z^j_{[s]}}(\psi,\varphi) - \|Tk_{\varphi(z^j)}\|_{n, q}.
\end{align*}
Clearly, $\varphi(z^j) = Az^j+b \to \infty$ as $j \to \infty$. Then, by Lemma \ref{lem-com1}, $\|Tk_{Az^j+b}\|_{n, q} \to 0$ as $j \to \infty$.

From this and \eqref{eq-in-1}, we get
$$
N > \|W_{\psi, \varphi} - T\| \geq \lim_{j\to\infty}\ell_{z^j_{[s]}}(\psi,\varphi) > M,
$$
which is a contradiction.

\textbf{Upper estimate.} We fix a sequence of positive numbers $(\lambda_j)_j \uparrow 1$ and, for each $j \in \NN$, put
$ C_j = C_{\lambda_j I_n}$, where $I_n$ is the unit $n \times n$ matrix, that is,
$$
C_jf(z) = C_{\lambda_j I_n}f(z) = f(\lambda_jz), \ f \in \calO(\CC^n), z \in \CC^n.
$$
By Corollaries \ref{cor-iv} and \ref{cor-co}, the operator $C_j$ is compact from $\calF^p(\CC^n)$ into itself and
$ \|C_j\| \leq \lambda_j^{\frac{-2n}{p}}$.
Let us denote by $I$ the identity operator on $\calF^p(\CC^n)$ and put $T_j = I - C_j$. Obviously, $\|T_j\| \leq 1 + \lambda_j^{\frac{-2n}{p}}$ for every $j \in \NN$.

For any $R>0$ and $j \in \NN$, using \eqref{eq-normf}, we have
\begin{align*}
&\|W_{\psi,\varphi}\|_e \leq \|W_{\psi,\varphi} - W_{\psi,\varphi} C_j\|
= \sup_{\|f\|_{n,p} \leq 1}\|W_{\psi,\varphi}(I - C_j)f\|_{n,q}\\
= & \sup_{\|f\|_{n,p} \leq 1} \left( \left(\dfrac{q}{2\pi}\right)^s \int_{\CC^s} |T_jf(\varphi(z))|^qe^{-\frac{q\left|z_{[s]}\right|^2}{2}} \|\psi(z_{[s]},\cdot)\|_{n-s,q}^q dA(z_{[s]}) \right)^{\frac{1}{q}} \\
\leq & \sup_{\|f\|_{n,p} \leq 1} \left( \left(\dfrac{q}{2\pi}\right)^s \int_{|z_{[s]}| \leq R} |T_jf(\varphi(z))|^qe^{-\frac{q\left|z_{[s]}\right|^2}{2}} \|\psi(z_{[s]},\cdot)\|_{n-s,q}^q dA(z_{[s]}) \right)^{\frac{1}{q}}\\
+ & \sup_{\|f\|_{n,p} \leq 1} \left( \left(\dfrac{q}{2\pi}\right)^s \int_{|z_{[s]}| > R} |T_jf(\varphi(z))|^qe^{-\frac{q\left|z_{[s]}\right|^2}{2}} \|\psi(z_{[s]},\cdot)\|_{n-s,q}^q dA(z_{[s]}) \right)^{\frac{1}{q}}\\
=&\; \calI(j,R) + \calJ(j,R).
\end{align*}

On one hand, for $\calJ(j,R)$, by Lemmas \ref{lem-F} and \ref{lem-Fpq}, we have
\begin{align*}
&\calJ(j,R) \leq  \sup_{|z_{[s]}| > R} \ell_{z_{[s]}}(\psi,\varphi) \\
& \times \sup_{\|f\|_{n,p} \leq 1} \left( \left(\dfrac{q}{2\pi}\right)^s \int_{|z_{[s]}| > R} |T_jf(Az+b)|^qe^{-\frac{q|Az+b|^2}{2}} dA(z_{[s]}) \right)^{\frac{1}{q}}\\
\end{align*}
\begin{align*}
\leq & \; |\textnormal{det}A_{[s]}|^{-\frac{2}{q}}\sup_{|z_{[s]}| > R} \ell_{z_{[s]}}(\psi,\varphi) \\
& \times \sup_{\|f\|_{n,p} \leq 1} \left( \left(\dfrac{q}{2\pi}\right)^s \int_{\CC^s} \left| T_jf(\zeta_{[s]}, b'_{[s]})\right|^qe^{-\frac{q\left(\left|\zeta_{[s]}\right|^2 + \left|b'_{[s]}\right|^2 \right)}{2}} dA(\zeta_{[s]}) \right)^{\frac{1}{q}}\\
=& \; |\textnormal{det}A_{[s]}|^{-\frac{2}{q}}\sup_{|z_{[s]}| > R} \ell_{z_{[s]}}(\psi,\varphi) \sup_{\|f\|_{n,p} \leq 1} \|T_jf(\cdot,b'_{[s]})\|_{s, q}e^{-\frac{\left|b'_{[s]}\right|^2}{2}} \\
\leq& \; |\textnormal{det}A_{[s]}|^{-\frac{2}{q}}\sup_{|z_{[s]}| > R} \ell_{z_{[s]}}(\psi,\varphi) \sup_{\|f\|_{n,p} \leq 1} \|T_jf\|_{n, q}\\
\leq& \; \left(\dfrac{q}{p}\right)^{\frac{n}{q}} |\textnormal{det}A_{[s]}|^{-\frac{2}{q}}\sup_{|z_{[s]}| > R} \ell_{z_{[s]}}(\psi,\varphi) \sup_{\|f\|_{n,p} \leq 1} \|T_jf\|_{n, p}\\
=& \; \left(\dfrac{q}{p}\right)^{\frac{n}{q}}  \|T_j\| |\textnormal{det}A_{[s]}|^{-\frac{2}{q}}\sup_{|z_{[s]}| > R} \ell_{z_{[s]}}(\psi,\varphi)\\
\leq& \left(\dfrac{q}{p}\right)^{\frac{n}{q}} \left(1 + \lambda_j^{\frac{-2n}{p}} \right)|\textnormal{det}A_{[s]}|^{-\frac{2}{q}}\sup_{|z_{[s]}| > R} \ell_{z_{[s]}}(\psi,\varphi).
\end{align*}

On the other hand, for $\calI(j,R)$, we have
\begin{align*}
&\calI(j,R) \leq \sup_{|z_{[s]}| \leq R} \ell_{z_{[s]}}(\psi,\varphi) \\
& \times \sup_{\|f\|_{n,p} \leq 1} \left( \left(\dfrac{q}{2\pi}\right)^s \int_{|z_{[s]}| \leq R} |T_jf(Az+b)|^qe^{-\frac{q|Az+b|^2}{2}} dA(z_{[s]}) \right)^{\frac{1}{q}}\\
\leq & \; \ell(\psi,\varphi) \sup_{\|f\|_{n,p} \leq 1} \max_{|z_{[s]}| \leq R}|T_jf(Az+b)| \left( \left(\dfrac{q}{2\pi}\right)^s \int_{\CC^s} e^{-\frac{q|Az+b|^2}{2}} dA(z_{[s]}) \right)^{\frac{1}{q}}\\
=& \; |\textnormal{det}A_{[s]}|^{-\frac{2}{q}} e^{-\frac{\left|b'_{[s]}\right|^2}{2}}\ell(\psi,\varphi) \sup_{\|f\|_{n,p} \leq 1} \max_{|z_{[s]}| \leq R}|T_jf(Az+b)|\\
\leq &  \; |\textnormal{det}A_{[s]}|^{-\frac{2}{q}} e^{-\frac{\left|b'_{[s]}\right|^2}{2}}\ell(\psi,\varphi) \sup_{\|f\|_{n,p} \leq 1} \max_{|\zeta_{[s]}| \leq R_1}|T_jf(\zeta_{[s]},b'_{[s]})| \\
\leq & \; |\textnormal{det}A_{[s]}|^{-\frac{2}{q}} e^{-\frac{\left|b'_{[s]}\right|^2}{2}} \ell(\psi,\varphi) \sup_{\|f\|_{n,\infty} \leq 1} \max_{|\zeta_{[s]}| \leq R_1}|T_jf(\zeta_{[s]},b'_{[s]})|,
\end{align*}
where
$$
R_1 = \max_{|z_{[s]}| \leq R} |A_{[s]}z_{[s]}+ b_{[s]}|,
$$
and the last inequality is due to the fact that $\|f\|_{n,\infty} \leq \|f\|_{n,p}$ for every $f \in \calF^p(\CC^n)$.

Now for each function $f(z) = \sum_{|i|=0}^{\infty} a_i z^i$ with $\|f\|_{n,\infty} \leq 1$, by the Cauchy inequality for Taylor coefficients, for every $r = (r_1,...,r_n) \in \RR_+^n$ and $i = (i_1,...,i_n) \in \NN_0^n$, we have
\begin{align*}
|a_i| &= \dfrac{1}{i!} \left| \dfrac{\partial^{|i|}f}{\partial^{i_1}z_1...\partial^{i_n}z_n}(0)\right|
 \leq \dfrac{\max\{|f(z)|: z \in D(0,r)\}}{r_1^{i_1}...r_n^{i_n}} \\
& \leq \dfrac{\max\{e^{\frac{|z|^2}{2}}: z \in D(0,r)\}}{r_1^{i_1}...r_n^{i_n}} = \dfrac{e^{\frac{r_1^2+...+r_n^2}{2}}}{r_1^{i_1}...r_n^{i_n}},
\end{align*}
where, as usual, $|i| = i_1 + ... + i_n$, $i! = i_1!...i_n!$ and $z^i = z_1^{i_1}...z_n^{i_n}$ and
$D(0,r) =\{z \in \CC^n: |z_1| \leq r_1,...,|z_n| \leq r_n\}$.

It implies that
$$
|a_i| \leq \inf_{r_1 > 0} \dfrac{e^{\frac{r_1^2}{2}}}{r_1^{i_1}} ... \inf_{r_n > 0} \dfrac{e^{\frac{r_n^2}{2}}}{r_n^{i_n}} = \left(\dfrac{e}{i_1}\right)^{\frac{i_1}{2}}...\left(\dfrac{e}{i_n}\right)^{\frac{i_n}{2}},
$$
with a convention that $\left(\frac{e}{t}\right)^{\frac{t}{2}} = 1$ when $t = 0$.

From this it follows that
\begin{align*}
&\calI(j,R) \leq |\textnormal{det}A_{[s]}|^{-\frac{2}{q}} e^{-\frac{\left|b'_{[s]}\right|^2}{2}} \ell(\psi,\varphi) \sup_{\|f\|_{n,\infty} \leq 1} \max_{|\zeta_{[s]}| \leq R_1}|(I-C_j)f(\zeta_{[s]},b'_{[s]})|\\
\leq & \; |\textnormal{det}A_{[s]}|^{-\frac{2}{q}} e^{-\frac{\left|b'_{[s]}\right|^2}{2}} \ell(\psi,\varphi) \sup_{\|f\|_{n,\infty} \leq 1} \max_{|\zeta_{[s]}| \leq R_1} \sum_{|i|=1}^{\infty}|a_i| (1-\lambda_j^{|i|}) \left|(\zeta_{[s]},b'_{[s]})^i\right| \\
\leq & \; |\textnormal{det}A_{[s]}|^{-\frac{2}{q}} e^{-\frac{\left|b'_{[s]}\right|^2}{2}} \ell(\psi,\varphi) (1-\lambda_j)\sum_{|i|=1}^{\infty}\left(\dfrac{e}{i_1}\right)^{\frac{i_1}{2}} \cdots \left(\dfrac{e}{i_n}\right)^{\frac{i_n}{2}} |i| R_2^{|i|},
\end{align*}
where $R_2 = \max\{R_1, |b'_{[s]}|\}$.

Consequently, for every $R>0$,
\begin{align*}
\|W_{\psi,\varphi}\|_e &\leq \limsup_{j \to \infty} \|W_{\psi,\varphi} - W_{\psi,\varphi} C_j\| \leq \limsup_{j \to \infty} (\calI(i,R) + \calI(j,R))\\
& \leq 2 \left(\dfrac{q}{p}\right)^{\frac{n}{q}} |\textnormal{det}A_{[s]}|^{-\frac{2}{q}}\sup_{|z_{[s]}| > R} \ell_{z_{[s]}}(\psi,\varphi).
\end{align*}
Letting $R \to \infty$ in this inequality we get the desired upper estimate for $\|W_{\psi,\varphi}\|_e$.
\end{proof}

From Lemma \ref{lem-as} and Theorem \ref{thm-essnorm} we obtain the following result for the case when $A$ is invertible.

\begin{cor}\label{cor-essnorm}
Let $1 < p \leq q < \infty$ and $W_{\psi,\varphi}: \calF^p(\CC^n) \to \calF^q(\CC^n)$ be a bounded weighted composition operator induced by a pair $(\psi, \varphi) \in \calW_q$ with $\varphi(z) = Az + b$ and $A$ is invertible. Then
$$
\limsup_{z \to\infty}m_{z}(\psi, \varphi) \leq \|W_{\psi,\varphi}\|_e \leq 2 |\textnormal{det}A|^{-\frac{2}{q}} \left(\dfrac{q}{p}\right)^{\frac{n}{q}} \limsup_{z \to\infty}m(\psi, \varphi).
$$
\end{cor}

In particular, Corollary \ref{cor-essnorm} contains the corresponding result in \cite[Theorem 3.7]{TK-17} as a particular case when $n = 1$.

\begin{rem}\label{rem-incl}
Suppose that $(\widehat{\psi},\widehat{\varphi})$ is another normalization of $(\psi,\varphi)$. Then by Lemma \ref{lem-UV}, there is an $s \times s$ unitary matrix $H$ such that
$\ell_{z_{[s]}}(\widehat{\psi},\widehat{\varphi}) = \ell_{Hz_{[s]}}(\widetilde{\psi},\widetilde{\varphi})$ for all $z_{[s]} \in \CC^s$. This implies that
$$
\ell(\widetilde{\psi},\widetilde{\varphi}) = \ell(\widehat{\psi},\widehat{\varphi})  \text{ and } \limsup_{z_{[s]} \to \infty} \ell_{z_{[s]}}(\widetilde{\psi},\widetilde{\varphi}) = \limsup_{z_{[s]} \to \infty} \ell_{z_{[s]}}(\widehat{\psi},\widehat{\varphi}).
$$
Moreover, in the case $0 < q < p < \infty$,
$
\ell_{z_{[s]}}(\widetilde{\psi},\widetilde{\varphi}) \in L^{\frac{pq}{p-q}}(\CC^s,dA)
$
if and only if
$
\ell_{z_{[s]}}(\widehat{\psi},\widehat{\varphi}) \in L^{\frac{pq}{p-q}}(\CC^s,dA).
$
Also
$$
\|\ell_{z_{[s]}}(\widetilde{\psi},\widetilde{\varphi})\|_{L^{\frac{pq}{p-q}}} = \|\ell_{z_{[s]}}(\widehat{\psi},\widehat{\varphi})\|_{L^{\frac{pq}{p-q}}}.
$$
By these assertions, our results in Theorems \ref{thm-bounded}, \ref{thm-compact}, \ref{thm-bd} and \ref{thm-essnorm} do not depend on the choice of a normalization $(\widetilde{\psi},\widetilde{\varphi})$ of $(\psi,\varphi)$.
\end{rem}

\bigskip

\textbf{Acknowledgement.} This article was carried out during the first-named author's stay at Division of Mathematical Sciences, School of Physical and Mathematical Sciences, Nanyang Technological University, as a postdoc fellow under the MOE's AcRF Tier 1 M4011724.110 (RG128/16). He would like to thank the institution for hospitality and support.

\end{document}